\documentclass[11pt]{article}%
\usepackage{float} % fix table inside context 
\usepackage{fullpage}
\usepackage{amsfonts}
\usepackage{amsmath}
\usepackage{amssymb}
\usepackage{amsbsy}
\usepackage{amsthm}
\usepackage{mathtools}
\usepackage{epsfig}
\usepackage{graphicx}
\usepackage{colordvi}
\usepackage{graphics}
\usepackage{color}
\usepackage{bm}
\usepackage{verbatim} 
\numberwithin{equation}{section}

\newtheorem{theorem}{Theorem}[section]
\newtheorem{thm}{Theorem}[section]
\newtheorem{lem}{Lemma}[section]

\newtheorem{lemma}{Lemma}[section]

\newtheorem{alg}[thm]{Algorithm}
\newtheorem{remark}[thm]{Remark}

\usepackage{algpseudocode}

\newcommand{\ee}{\end{equation}}
\newcommand{\bea}{\begin{eqnarray}}
\newcommand{\eea}{\end{eqnarray}}
\newcommand{\beas}{\begin{eqnarray*}}
	\newcommand{\eeas}{\end{eqnarray*}}

\newcommand{\vertiii}[1]{{\left\vert\kern-0.25ex\left\vert\kern-0.25ex\left\vert #1
		\right\vert\kern-0.25ex\right\vert\kern-0.25ex\right\vert}}

\newcommand{\normiii}[1]{{\left\vert\kern-0.25ex\left\vert\kern-0.25ex\left\vert #1
		\right\vert\kern-0.25ex\right\vert\kern-0.25ex\right\vert}}

\newcommand{\bu}{u}
\newcommand{\bv}{ v}

\newcommand{\be}{ e}

\newcommand{\bV}{ V}

\begin{document}
\title{Enhancing nonlinear solvers for the Navier-Stokes equations with continuous (noisy) data assimilation}

\author{
Bosco Garc{\' i}a-Archilla \thanks{\small Departamento de Matem\'atica Aplicada
II, Universidad de Sevilla, Sevilla, Spain. Research is supported by grants
 PID2021-123200NB-I00 and PID2022-136550NB-I00  funded by MCIN/AEI/
10.13039/501100011033 and by ERDF A way of making Europe, by
the European Union.  (bosco@esi.us.es)}
\and
Xuejian Li\thanks{\small School of Mathematical and Statistical Sciences, Clemson University, Clemson, SC, 29364 (xuejial@clemson.edu)}  
\and
Julia Novo \thanks{\small Departamento de
Matem\'aticas, Universidad Aut\'onoma de Madrid, Spain. Research is supported
by grant PID2022-136550NB-I00  funded by MCIN/AEI/
10.13039/501100011033 and by ERDF A way of making Europe, by
the European Union. (julia.novo@uam.es)}
\and
Leo G. Rebholz\thanks{\small School of Mathematical and Statistical Sciences, Clemson University, Clemson, SC, 29364.  Research,
is partially supported by NSF grant DMS 2152623  (rebholz@clemson.edu)}
		}

	\maketitle
	
	\begin{abstract}{
		We consider nonlinear solvers for the incompressible, steady (or at a fixed time step for unsteady) Navier-Stokes equations in the setting where  partial measurement data of the solution is available.  The measurement data is incorporated/assimilated into the solution through a nudging term addition to the the Picard iteration that penalized the difference between the coarse mesh interpolants of the true solution and solver solution, analogous to how continuous data assimilation (CDA) is implemented for time dependent PDEs.  This was considered in the paper [Li et al. {\it CMAME} 2023], and we extend the methodology by improving the analysis to be in the $L^2$ norm instead of a weighted $H^1$ norm where the weight depended on the coarse mesh width, and to the case of noisy measurement data.  For noisy measurement data, we prove that the CDA-Picard method is stable and convergent, up to the size of the noise.  Numerical tests illustrate the results, and show that a very good strategy when using noisy data is to use CDA-Picard to generate an initial guess for the classical Newton iteration.
		}
\end{abstract}

	\section{Introduction}
	We consider herein nonlinear solvers for the Navier-Stokes equations (NSE), which are widely used for modeling incompressible fluid flow.  In particular, we study the steady state NSE, which are given on a domain $\Omega \subset \mathbb{R}^d, d=2,3$ by
	\begin{equation}\label{NS}
		\left\{\begin{aligned}
			-\nu \Delta u+u\cdot\nabla u+ \nabla p&={f} \quad \text{in}~\Omega,\\
			\nabla\cdot {u}&=0\quad \text{in}~\Omega,\\
			{u}&=0 \quad \text{on}~~\partial\Omega,
		\end{aligned}\right.
	\end{equation}
	where $u$ is the fluid velocity, $p$ is the pressure, $\nu$ is the kinematic viscosity,  and ${f}$ is an external forcing term. The Reynolds number $Re:=\frac{1}{\nu}$ is a parameter that describes how complex a flow is; higher $Re$ is associated with more complex physics and non-unique steady solutions.
	While our study is restricted to the steady system \eqref{NS} with homogenous Dirichlet boundary conditions, the results are extendable to solving the time dependent NSE at a fixed time step in a time stepping scheme as well as nonhomogeneous mixed Dirichlet/Neumann boundary conditions. 
	
	Probably the most common nonlinear iteration for solving \eqref{NS}  is  the Picard iteration, which is given by (suppressing the boundary conditions)
	\begin{align*}
		-\nu \Delta u_{k+1}+u_k\cdot\nabla u_{k+1}+ \nabla p_{k+1}&={f}, \\
		\nabla\cdot {u}_{k+1}&=0.
	\end{align*}
	The Picard iteration is globally convergent under a smallness condition on the Reynolds number / problem parameters $\alpha:=M\nu^{-2} \|f \|_{H^{-1}}<1$, where $M$ is a domain-size dependent constant arising from Sobolev/Ladyzhenskaya inequalities \cite{Laytonbook,GR86,PRX19}.  Provided $\alpha<1$, the constant $\alpha$ is an upper bound on Picard's $H^1$ linear convergence rate \cite{GR86} and is believed reasonably sharp.  Moreover, $\alpha<1$ also implies that steady NSE solutions are unique.
	
	In \cite{LHRV23}, the CDA-Picard iteration  was proposed, in order to improve convergence in the setting where solution data is available, e.g. from measurements or observed data.  The CDA-Picard algorithm takes the form
	\begin{align*}
		-\nu \Delta u_{k+1}+u_k\cdot\nabla u_{k+1}+ \nabla p_{k+1} + \mu I_H(u_{k+1} - u)&={f}, \\
		\nabla\cdot {u}_{k+1}&=0,
	\end{align*}
where $\mu>0$ is a (user-defined) nudging parameter, and $I_H$ is an appropriate coarse mesh interpolant (defined precisely in section 2).  We note as is typical in CDA, the true solution $u$ remains unknown, but $I_H(u)$ is known from given or observed solution data.   The general idea is that, if one knows the solution at some points in space, then one should be able to improve convergence and robustness of a solver.  Indeed, in \cite{LHRV23}, it is proven analytically and illustrated numerically how CDA can accelerate convergence and enlarge the convergence basin for Picard.  
	
The use of CDA above is based on the Azouani, Olson, and Titi algorithm from 2014 \cite{AOT14}, and as noted in \cite{LHRV23} 
its intent is for time dependent problems.  It has been used on a wide variety of problems including Navier-Stokes equations and turbulence \cite{AOT14,DMB20,FLT19}, Benard convection \cite{FJT15}, a planetary geostrophic model \cite{FLT16}, the Cahn-Hilliard equation \cite{DR22}, and many others.  CDA has gained immense interest since its development and recent research includes sensitivity analyses \cite{CL21}, numerical analyses \cite{IMT20,LRZ19,RZ21,DR22,GNT18,JP23}, efficient nudging methods \cite{RZ21}, effects of noisy data \cite{Foias_Mondaini_Titi_2016}, and using it for parameter recovery \cite{CHL20}.  To our knowledge, the CDA-Picard method from \cite{LHRV23} is the first time CDA was used to improve efficiency and robustness of nonlinear solvers for steady PDEs.

The purpose of this paper is to extend the study of CDA-Picard in two ways.  First, we provide an improved convergence analysis for CDA-Picard based on the $L^2$ norm, instead of the weighted $H^1$ norm that depends on the coarse mesh width $H$ used in the original analysis from \cite{LHRV23}.  Second, we extend the study of CDA-Picard to the more realistic case of noisy data, both analytically and numerically.  We prove that for the case of noisy data, CDA-Picard is stable and convergent provided enough measurement data is known.  However, the convergence is only for the $L^2$ residual (which is the same as the difference in successive iterates if one considers CDA-Picard as a fixed point iteration, i.e. $\| u_{k+1} - u_k \|_{L^2} = \| g_{picard}(u_k) - u_k \|_{L^2}$), as the $L^2$ error $\| u_{nse} - u_k \|$ is found to scale with the size of the noise in the data.  Our analysis holds for any Reynolds number and forcing, i.e. it is valid in the case of non-unique steady NSE solutions, but the amount of required measurement data increases as $Re$ and other problem parameters increase.  However, we note that proofs for larger $Re$ require different techniques and not surprisingly have slightly worse scalings.  In addition to the analysis, we give results of numerical tests on two challenging test problems which illustrate the analytical results.  Finally, we show numerically that a very effective solution method when given noisy data is to run CDA-Picard until its residual is sufficiently small (so that its error is as small as possible, which is on the order of the noise level), and then use its final iterate as the initial guess to usual Newton; we call this method CDA-Picard + Newton.

The paper is organized as follows.  Section 2 gives preliminaries, sets notation for the paper, and gives background on CDA-Picard.  Section 3 gives the improved analysis for CDA-Picard using the $L^2$ norm instead of the weighted $H^1$ norm, and section 4 analyzes CDA-Picard for the case of noisy data.  Here we analyze stability, residual convergence, and error, both for smaller $Re$ and larger $Re$.  Section 5 contains numerical results which illustrate the theory and also test CDA-Picard + Newton.  Finally, conclusions are drawn in Section 6.

	\section{Preliminaries}\label{Pre}
	
	The domain $\Omega$ will be an open connected set in $\mathbb{R}^d$ (d=2 or 3) that is a convex polygon or has smooth boundary.  We use $(\cdot,\cdot)$ to denote the $L^2$ inner product that induces the $L^2$ norm $\|\cdot\|$, and all other norms will have appropriate subscripts.
	
	We denote the natural NSE function spaces for the velocity and pressure by 
	\begin{align}
		&Q:=\{v\in {L}^2(\Omega): \int_{\Omega}vdx=0\},\\
		& X:=
		\{v\in H^1\left(\Omega\right): v=0~~\text{on}~ \partial\Omega\}.
	\end{align}
		The divergence free subspace is denoted by
		\[
		V:=
		\{v\in X:  (\nabla \cdot v,q)=0\ \forall q\in Q\}.
		\]
			
Recall the Poincare inequality holds on $X$: there exists a constant $C_P$ dependent only on the domain satisfying $\| \phi \| \le C_P \|\nabla \phi \|$ for all $\phi\in X$.

\subsection{Steady NSE preliminaries}
	
	The weak form of the NSE (\ref{NS}) is to find $(u,p)\in X\times Q$ such that 
	\begin{equation}\label{weakNavier}
		\left\{\begin{aligned}
				\nu \left(\nabla u, \nabla v\right)+b\left({u},{u},v\right)+(p,\nabla\cdot v)&=\left\langle{f},v\right\rangle~~\forall v\in X,\\
			(\nabla\cdot{u},q)&=0~~\forall q\in Q,
		\end{aligned}\right.
	\end{equation}
	%\begin{equation}
	%	\label{weakNavier}
	%	a\left({u},v\right)+b\left({u},{u},v\right)+(p,\nabla\cdot v)=\left\langle{f},v\right\rangle~~\forall v\in X,~~(\nabla\cdot{u},q)=0~~\forall q\in Q,
	%\end{equation}
	where 
	\begin{align*}
		b(u,w,v)&= \left((u\cdot\nabla)w, v\right) ~~\forall u,w,v\in X.
	\end{align*}
Since the pair $(X,Q)$ satisfies the inf-sup condition, we can instead consider the equivalent system \cite{GR86}: Find $u\in V$ satisfying
	\begin{equation}\label{wd}
		\nu \left(\nabla u, \nabla v\right)+b\left({u},{u},v\right)=\left\langle{f},v\right\rangle~~\forall v\in V.
	\end{equation}

It is important to note that $b(u,v,v)=0$ for all $u,v\in X$.  The following inequalities hold for $b$ \cite{Laytonbook,temam}: there exists a constant $M$ dependent only on the domain $\Omega$ such that
	\begin{align}
		b(u,w,v)&\leq M\|u\|^{\frac{1}{2}}\|\nabla u\|^{\frac{1}{2}}\|\nabla w\|\|\nabla v\|, \label{bbound1} \\
b(u,w,v)&\leq M \|\nabla u\|\|\nabla w\|\|\nabla v\|.\label{bbound2}
	\end{align}

Although this paper considers analysis in $V$, in discretizations using $(X_h,Q_h)\subset (X,Q)$ but where $\nabla \cdot X_h \not\subset Q_h$, the analysis of this paper will not immediately transfer to such discretizations due to the trilinear term.  A fix for this is to instead use the classical skew-symmetric form of the trilinear term,
\[
b^*(u,v,w)=b(u,v,w)+\frac12 ((\nabla \cdot u)v,w)
\]
and also grad-div stabilization.  These modifications will allow all of our results to extend to the case of mixed finite elements whose velocities are only discretely divergence free.  All of our computations use divergence-free Scott-Vogelius elements where $\nabla \cdot X_h \subset Q_h$ does hold.

 Recall the classical  well-posedness result for equation (\ref{wd}) \cite{Laytonbook,temam}:
	
	\begin{lemma}
		Let $\alpha=M\nu^{-2}\|f\|_{-1}$. For any $f\in H^{-1}$ and $\nu>0$,	solutions to  (\ref{wd}) exist and satisfy
		\begin{align}\label{Pri}
			\|\nabla u\|\leq \nu^{-1}\|f\|_{-1}.
		\end{align}
		Furthermore, if $\alpha<1$, 
		the solution is unique.
	\end{lemma}
	The restriction $\alpha< 1$ is usually referred to as the small data condition for steady NSE, but we refer to it herein as a smallness condition since we use the term `data' to mean observations or measurements (and not the PDE parameters).

We will assume throughout this work that NSE solutions we consider satisfy
$$
\|\nabla \bu\|_\infty =: K_1 < \infty,
$$
and also that $f$ is sufficiently regular to allow $K_1$ to be finite.

\subsection{Picard and Newton iterations}

Herein we will mostly study the Picard iteration and variations of it, but Newton will also be used in the numerical tests.  These iterations can be written in their $V$-formulations as follows: Find $u_{k+1}\in V$ satisfying for all $v\in V$ that
		\begin{align}
				\label{PicardO}a\left({u}_{k+1},v\right)+b\left({u}_{k},{u}_{k+1},v\right)&=\left\langle{f},v\right\rangle~~ \text{(Picard)},\\
				\label{NewtonO}a\left({u}_{k+1},v\right)+b\left({u}_{k},{u}_{k+1},v\right)+b\left({u}_{k+1},{u}_{k},v\right)& =\left\langle{f},v\right\rangle +b\left({u}_{k},{u}_{k},v\right)~~ \text{(Newton)}.
		\end{align}
	
	We recall some basic results about the Picard iteration (\ref{PicardO}) from \cite{GR86}.
	
		\begin{lemma}\label{PLemma}
		The Picard method (\ref{PicardO}) is unconditionally stable. Furthermore, if $\alpha<1$
		then the sequence $\{u_k\}$ generated by Picard converges to the NSE solution $u$ as $k\to \infty$ with an $\alpha$-linear rate.
	\end{lemma}

\subsection{CDA and CDA-Picard}

We now give CDA preliminaries and define CDA-Picard.  Denote by $\tau_H(\Omega)$ a coarse mesh of $\Omega$ used to represent an interpolant of the true solution using measurements or observables.  We assume that $I_H$ satisfies the usual properties as other CDA applications: there exists a constant $C_I$ independent of $H$ satisfying
		\begin{align}
			\label{interpolationi}	\|I_Hv-v\|\leq C_IH \| \nabla v\|~~~~\forall v\in X,\\
			\label{interpolationi2}	\|I_Hv\|\leq C_I \|v\|~~~~\forall v\in X. 
		\end{align}
Some examples of such interpolation operators include Bernardi-Girault \cite{BG98}, Scott–Zhang \cite{SZ90}, and the $L^2$ projection onto piecewise constants \cite{fem:book:ern:guermond}.

%		
%Although our analysis is general to $(X,Q)$, in practice one uses finite dimensional (finite element) subspaces $(X_h,Q_h)$ defined on a fine mesh $\tau_h$.  For the implementation details discussed here and performed in numerical tests to follow, we require that nodes of $\tau_H$ also be nodes of $\tau_h$, and $h \ll H$.
%

The CDA-Picard iteration is given in its $V$-formulation by
\begin{align}
			\label{CDAPicard}	&a\left({u}_{k+1},v\right)+b\left({u}_{k},{u}_{k+1},v\right)+\mu (I_Hu_{k+1}-I_H  u,I_H v)=\left\langle{f},v\right\rangle~~\text{(CDA-Picard)}.
		\end{align}
The term $\mu I_H(u_{k+1}-u)$ is a CDA nudging (penalty) driving the Step $k+1$ solution $u_{k+1}$ towards to the observations, and $\mu>0$ is a relaxation parameter that emphasizes the observations' accuracy.  We note that CDA-Picard uses a type of variational crime on the nudging term, by using $I_H$ in the second argument of the nudging terms.  While this is consistent if $I_H$ is an $L^2$ projection, in general it is not consistent but using $I_H$ additionally on the test function is key to allowing for less restrictions on parameters and no upper bound on $\mu$ \cite{GNT18, GN20, RZ21}.  

Convergence of CDA-Picard is proven in \cite{LHRV23} in terms of the following weighted $H^1$ norm.
	\begin{align}\label{d1}
		\|v\|_{*}=\left(\|\nabla v\|^2+\frac{1}{2C_I^2H^2}\| v\|^2\right)^{\frac{1}{2}}.
\end{align}
The result is given next.  It shows that convergence is guaranteed for any $\alpha$ provided enough data is available, and that the convergence rate $O(H^{1/2}\alpha)$ is improved by including data measurements.

\begin{theorem}\label{Pcac} [Convergence of CDA-Picard \cite{LHRV23}]
Let $u$ be a steady NSE solution and suppose $I_H(u)$ is known.  If $\sqrt{2}C_IH\alpha^2<1$ and $\mu\geq \frac{\nu}{4C_I^2H^2}$, then for any initial guess $u_0$ the CDA-Picard iteration \eqref{CDAPicard} converges to $u$ linearly with rate (at least) $\sqrt{\sqrt{2} C_IH}\alpha$:
	\begin{align}\label{CDAPC}
		\|u - u_{k+1}\|_{*}&\leq \sqrt{ \sqrt{2}C_IH}\alpha\|u - u_{k}\|_{*}.
	\end{align}
\end{theorem}

\section{An improved analysis for the CDA-Picard iteration}\label{CDA-Picard}

We give below a new convergence analysis for CDA-Picard under the assumption that the data measurements are accurate (the next section considers the case of noisy measurement data).  While the result from Theorem \ref{Pcac} shows that data improves convergence, the *-norm is not an ideal norm to use because it is a weighted $H^1$ norm where the weight depends on $H$.  So decreasing $H$ will improve the contraction ratio, but at the same time the norm is closer to the $L^2$ norm and further from the $H^1$ norm.  Hence it would be better to show this kind of convergence result in a norm not dependent on $H$, and below we give a new convergence results for CDA-Picard in the $L^2$ norm.

%
%
%Although it is known that for $\alpha<1$ and any $\bff \in H^{-1}$ there exists a unique solution of \eqref{eq:NS} satisfying
%$$
%\|\nabla \bu\|_0\le \nu^{-1}\|\bff \|_{-1},
%$$
%we will assume that $\bff$ is such that there exists a unique solution of \eqref{eq:NS} for which 
%

\begin{theorem}\label{Pcac2} [$L^2$ Convergence of CDA-Picard]
Let $u$ be a steady NSE solution and suppose $I_H(u)$ is known.  Set $\gamma:=\frac{1}{K_1}\min\left\{ \frac{\nu}{ C_I^{2}H^{2}},\mu\right\}$, and suppose $\mu> 2 K_1$ and $H< \frac{\nu^{1/2}}{(2K_1)^{1/2}C_I}$.  Then it holds
	\begin{align}\label{CDAPC2}
\|\bu_{k+1}-\bu\|\le \frac{\sqrt{2}}{\sqrt{\gamma}} \|\bu_k-\bu\|.
	\end{align}
Since $\gamma>2$, CDA-Picard converges linearly in $L^2$ and the convergence is faster whenever $H$ decreases and $\mu$ increases.
\end{theorem}

\begin{remark}
In the finite element setting where subspace $(X_h,Q_h)\subset (X,Q)$ but where $\nabla \cdot X_h \not\subset Q_h$ (e.g. Taylor-Hood elements), the trilinear term $b(u,v,w)$ needs to be modified by a consistent skew-symmetrization (e.g. adding $\frac12 ((\nabla \cdot u)v,w)$) for stability or convergence results to hold.  Provided grad-div stabilization is also used, the result above will still hold with just slight modifications.  In this case, one reaches the bound \eqref{CDAPC2} with norm 
$||\cdot||_{\square}$ instead of $||\cdot||$ where $||v||_{\square}=\left(||v||^2+||\nabla \cdot v||^2\right)^{1/2}$. Our numerical tests use divergence-free Scott-Vogelius elements where $\nabla \cdot X_h \subset Q_h$ does hold.
\end{remark}

\begin{proof}
Denoting by $\be_k=\bu-\bu_k$ and arguing as in \cite{LHRV23} it is easy to get
$$
\nu \|\nabla \be_{k+1}\|^2+\mu \|I_H \be_{k+1}\|^2=-b(\be_k,\bu,\be_{k+1}).
$$
The nonlinear term treatment is a key difference from the analysis of \cite{LHRV23}, and begins with H\"older's inequality with $L^2-L^{\infty}-L^2$ followed by Young's inequality,
\begin{eqnarray*}
\nu \|\nabla \be_{k+1}\|^2+\mu \|I_H \be_{k+1}\|^2 & \le & \|\be_k\|\|\nabla \bu\|_{\infty}\|\be_{k+1}\| \\
&\le& \frac{1}{2}\|\nabla \bu\|_{\infty}\|\be_{k+1}\|^2+\frac{1}{2}\|\nabla \bu\|_{\infty}\|\be_{k}\|^2\\
&\le& \frac{1}{2}K_1\|\be_{k+1}\|^2+\frac{1}{2}K_1\|\be_{k}\|^2.
\end{eqnarray*}
Denoting $L=\frac{1}{2}K_1$ and arguing as in \cite{GNT18} we can write
$$
L\|\be_{k+1}\|^2\le 2L \|I_H\be_{k+1}\|^2+2L\|(I-I_H)\be_{k+1}\|^2.
$$
Since by assumption $\mu-2L\ge \mu/2$ (which gives $\mu \ge 4L$), we get the bound
\begin{eqnarray*}
\nu \|\nabla \be_{k+1}\|^2+\frac{\mu}{2}\|I_H \be_{k+1}\|^2-2L\|(I-I_H)\be_{k+1}\|^2\le \frac{1}{2}K_1\|\be_{k}\|^2.
\end{eqnarray*}
Arguing again as in \cite{GNT18}, applying \eqref{interpolationi} and using the assumption on $H$ we find
$$
\nu \|\nabla \be_{k+1}\|^2-2L\|(I-I_H)\be_{k+1}\|^2\ge \nu \|\nabla \be_{k+1}\|^2-2L C_I^2H^2\|\nabla \be_{k+1}\|^2
\ge \frac{\nu}{2}\|\nabla \be_{k+1}\|^2.
$$
Combining these estimates now yields
\begin{eqnarray*}
{\nu} \|\nabla \be_{k+1}\|^2+{\mu}\|I_H \be_{k+1}\|^2\le K_1\|\be_{k}\|^2.
\end{eqnarray*}

The bound \eqref{interpolationi}  gives us
\begin{eqnarray*}
\frac{\nu}{C_I^{2}H^{2}} \|(I-I_H)\be_{k+1}\|^2+{\mu }\|I_H \be_{k+1}\|^2\le{\nu} \|\nabla \be_{k+1}\|^2+{\mu }\|I_H \be_{k+1}\|^2\le K_1\|\be_{k}\|^2,
\end{eqnarray*}
%and so after setting
%\begin{equation}\label{gama}
%\gamma:=\frac{1}{K_1}\min\left\{ \frac{\nu}{ C_I^{2}H^{2}},\mu\right\},
%\end{equation}
and from the definition of $\gamma$ and taking into account
$$
\|I_H \be_{k+1}\|^2+\|(I-I_H)\be_{k+1}\|^2\ge\frac{1}{2} \|\be_{k+1}\|^2,
$$
we conclude that
$$
\frac{\gamma}{2}\|\be_{k+1}\|^2\le \|\be_k\|^2.
$$
Reducing implies
$$
\|\bu_{k+1}-\bu\|\le \frac{\sqrt{2}}{\sqrt{\gamma}} \|\bu_k-\bu\|,
$$
which is the stated estimate.
\end{proof}

\section{CDA-Picard with noisy data}

We consider now the case where the data used for nudging has noise (error).  That is, instead of knowing $u(x_i)$ at all the measurement points, we instead know $u(x_i)+\epsilon_i$.  Hence we will write $I_H \epsilon$ to denote the coarse mesh interpolant constructed from $\{ \epsilon_i \}_{i=1}^N$.  Typically, these errors can be considered random, and may result for example from errors in the measurement process.  In practice one may even know the relative size or distribution of the errors.  

We denote by $\epsilon\in V$ any divergence-free function such that $I_H\epsilon (x_i)=\epsilon_i$, and assume that 
\begin{equation}\label{uhat}
\| \nabla (u+\epsilon) \| \le (1+\delta) \| \nabla u \|,
\end{equation}
and thus one would expect $\delta$ to be small since it represents the size of the noise relative to the size of the solution.  It will be useful in our analysis to define 
\begin{equation}\label{alfadelta}
\alpha_{\delta} = (1+\delta)\alpha=(1+\delta)M\nu^{-2}\|f\|_{-1},
\end{equation} and 
\begin{equation}\label{cdelta}
C_{\delta} = 2\sqrt{1 + (1+\delta)^2}.
\end{equation}

\begin{alg}[CDA-Picard for noisy data]
The CDA-Picard iteration with noise is defined by: find $\bu_{k+1}\in V$ such that
\begin{eqnarray}\label{eq:pi_grad_div}
\nu(\nabla \bu_{k+1}, \nabla \bv)+b(\bu_k,\bu_{k+1},\bv)+\mu (I_H \bu_{k+1}-I_H (\bu + \epsilon),I_H \bv)=(f,\bv) \quad \forall \bv\in \bV.
\end{eqnarray}
\end{alg}

\subsection{Stability}
Let us denote by
 \begin{equation}\label{lambdahat}
 \hat \lambda = \min \bigg\{ \frac{\nu}{4 C_I^2 H^2},\ \frac{\mu}{2} \bigg\}.
 \end{equation}
We now prove two stability results.  The first is for any data, while the second one is sharper but requires a smallness condition of the PDE parameters  (not the measurement data).

\begin{lem}[Stability result 1] \label{lem1}
For any $\nu>0$, $\alpha>0$ and $\mu>0$ and $\hat\lambda$ defined in \eqref{lambdahat} the $k+1$ iterate satisfies the bound
\begin{equation}
\frac{\nu}{2} \| \nabla u_{k+1} \|^2  + \hat \lambda \| u_{k+1} \|^2 \le \nu^{-1} \| f \|_{-1}^2 + 2\mu C_I^2 C_P^2 \nu^{-2} \|f \|_{-1}^2 +  2 \mu \| I_H \epsilon \|^2. \label{bound1}
\end{equation}
Furthermore, if $\mu \le \frac{\nu}{2C_I^2 H^2}$, then 
\begin{equation}
\| u_{k+1} \|^2 \le 2\mu^{-1} \nu^{-1} \| f \|_{-1}^2 + 4 C_I^2 C_P^2 \nu^{-2} \| f\|_{-1}^2 + 4 \| I_H \epsilon \|^2. \label{bound1a}
\end{equation}
\end{lem}

\begin{proof}
Taking $v=u_{k+1}$ in \eqref{eq:pi_grad_div} vanishes the nonlinear term and leaves
\[
\nu \| \nabla u_{k+1} \|^2 + \mu \| I_H u_{k+1} \|^2 = (f,u_{k+1}) + \mu (I_H u,I_H u_{k+1}) + \mu (I_H \epsilon, I_H u_{k+1}).
\]
Using the definition of the $H^{-1}$ norm on the forcing term and Cauchy-Schwarz on the other right hand side terms, followed by applying Young's inequality on all of them, we obtain
\[
\nu \| \nabla u_{k+1} \|^2 + \mu \| I_H u_{k+1} \|^2 \le \nu^{-1} \| f \|_{-1}^2 + 2\mu \| I_H u \|^2 + 2 \mu \| I_H \epsilon \|^2.
\]
Breaking up the viscous term on the left hand side into two pieces, using property \eqref{interpolationi} of $I_H$ on this and property \eqref{interpolationi2} on the second right hand side term and then Poincar\'e, we get that
\[
\frac{\nu}{2} \| \nabla u_{k+1} \|^2 + \frac{\nu}{2 C_I^2 H^2} \| u_{k+1} - I_H u_{k+1}  \|^2  + \mu \| I_H u_{k+1} \|^2 \le \nu^{-1} \| f \|_{-1}^2 + 2\mu C_I^2 C_P^2  \| \nabla u \|^2 + 2 \mu \| I_H \epsilon \|^2.
\]
With the definition of $\hat\lambda$ \eqref{lambdahat} we get from the triangle inequality a lower bound on the left hand side which provides the bound
\[
\frac{\nu}{2} \| \nabla u_{k+1} \|^2  + \hat \lambda  \| u_{k+1} \|^2 \le \nu^{-1} \| f \|_{-1}^2 + 2\mu C_I^2 C_P^2 \nu^{-2} \|f \|_{-1}^2 +  2 \mu \| I_H \epsilon \|^2,
\]
thanks to  the bound on $u$ from \eqref{Pri}.

The bound \eqref{bound1a} can be deduced from \eqref{bound1} by dropping the first left hand side term, reducing the min term to $\frac{\mu}{2}$, multiplying both sides by $\frac{2}{\mu}$ and reducing.

\end{proof}

The following stability bound assumes the PDE parameters and noise are sufficiently small, and shows that the stability bound is almost the same as the steady NSE solution bound \eqref{Pri}, up to an O(1) constant factor.

\begin{lem}[Stability result 2] \label{lem2}
Let $\alpha_\delta$ and $C_{\delta}$ be as in \eqref{alfadelta}, \eqref{cdelta} with $\alpha_{\delta}<\frac13$.  Then for $k$ sufficiently large,
\[
\| \nabla u_{k+1} \| \le 2\nu^{-1} \| f \|_{-1} \sqrt{ 1 + (1+\delta)^2 } = C_{\delta} \nu^{-1} \| f\|_{-1}.
\]
\end{lem}
\begin{proof}
Set $\hat u = u + \epsilon$, and choose $v=u_{k+1} - \hat u$ in \eqref{eq:pi_grad_div} to get 
\[
\frac{\nu}{2} \left( \| \nabla u_{k+1} \|^2 - \| \nabla \hat u \|^2 + \| \nabla (u_{k+1}-\hat u) \|^2 \right) + \mu \| I_H (u_{k+1} - \hat u) \|^2
=
(f,u_{k+1} - \hat u) + b(u_k,u_{k+1},\hat u),
\]
using the skew-symmetric property of $b$.  Using the assumed bound on $\hat u$ \eqref{uhat}, \eqref{bbound2}
and \eqref{Pri}, we obtain 
\begin{align*}
\frac{\nu}{2} \| \nabla u_{k+1} \|&^2  + \frac{\nu}{2} \| \nabla (u_{k+1}-\hat u) \|^2 + \mu \| I_H (u_{k+1} - \hat u) \|^2 \\
& \le \frac{\nu}{2} \| \nabla \hat u \|^2 + \| f\|_{-1} \|\nabla (u_{k+1}-\hat u) \| + M \| \nabla u_k \| \| \nabla u_{k+1} \| \| \nabla \hat u \|\\
& \le \frac{\nu^{-1}(1+\delta)^2}{2}\| f \|_{-1}^2  + \frac{\nu^{-1}}{2} \| f\|_{-1}^2 + \frac{\nu}{2} \|\nabla (u_{k+1}-\hat u)\|^2 + M \nu^{-1} \| f \|_{-1}(1+\delta) \| \nabla u_k \| \| \nabla u_{k+1} \| \\
& \le \frac{\nu^{-1}(1+\delta)^2}{2}\| f \|_{-1}^2  + \frac{\nu^{-1}}{2} \| f\|_{-1}^2 + \frac{\nu}{2} \|\nabla (u_{k+1}-\hat u)\|^2 + \nu \alpha_{\delta} \| \nabla u_k \| \| \nabla u_{k+1} \|,
\end{align*}
with the last step thanks to the definition of $\alpha_{\delta}$ \eqref{alfadelta}.  This reduces to
\[
\nu \| \nabla u_{k+1} \|^2 +  2\mu \| I_H (u_{k+1} - \hat u) \|^2
\le 
\nu^{-1} \| f \|_{-1}^2 (1 + (1+\delta)^2 )  + 2\nu \alpha_{\delta} \| \nabla u_k \| \| \nabla u_{k+1} \|.
\]
Next, drop the second term on the left hand size, and then apply Young's inequality to the last right hand side term via
\[
2\nu \alpha_{\delta} \| \nabla u_k \| \| \nabla u_{k+1} \| \le \frac{\nu}{2} \| \nabla u_{k+1} \|^2 + 2\nu \alpha_{\delta}^2 \| \nabla u_k \|^2,
\]
to obtain the bound
\[
\frac{\nu}{2} \| \nabla u_{k+1} \|^2 \le 
\nu^{-1} \| f \|_{-1}^2 (1 + (1+\delta)^2 )  +  2\nu \alpha_{\delta}^2 \| \nabla u_k \|^2.
\]
This reduces to
\begin{equation}
\| \nabla u_{k+1} \|^2 \le 
2\nu^{-2} \| f \|_{-1}^2 (1 + (1+\delta)^2 )  +  4 \alpha_{\delta}^2 \| \nabla u_k \|^2. \label{bound10}
\end{equation}
Elementary real analysis tells us that if $a_{n+1} \le b a_n + c$, where $a_i (i=0,1,...,n+1), b, c>0$, then $a_{n+1} \le b^{n+1}a_0 + c \frac{1-b^n}{1-b}$.  
%Since $\alpha_{\delta}<\frac12$ is assumed, 
We can apply this result to \eqref{bound10}, yielding
\begin{align*}
\| \nabla u_{k+1} \|^2 & \le (4 \alpha_{\delta}^2)^{k+1} \| \nabla u_0 \|^2 + 2\nu^{-2} \| f \|_{-1}^2 (1 + (1+\delta)^2 ) \frac{ 1 - (4 \alpha_{\delta}^2)^k }{1 - 4 \alpha_{\delta}^2}.
%\\
%& \le (4 \alpha_{\delta})^{k+1} \| \nabla u_0 \|^2 + 2\nu^{-2} \| f \|_{-1}^2 (1 + (1+\delta)^2 ) \frac{ 1  }{1 - (4 \alpha_{\delta})^{k+2}}
\end{align*}
Since $\alpha_{\delta}<\frac13$, then for $k$ sufficiently large, we obtain
\[
\| \nabla u_{k+1} \| \le 2\nu^{-1} \| f \|_{-1} \sqrt{ 1 + (1+\delta)^2 }.
\]
This completes the proof.

\end{proof}

\subsection{Convergence of the CDA-Picard with noisy data iteration residuals}

We now discuss convergence of the CDA-Picard iteration with noisy data.  We are able to essentially mimic the proofs of CDA-Picard for accurate data, to prove that the residuals of the method for noisy data converges.  Note that we can write $u_{k+1}-u_k = g_{picard} (u_k)-u_k$ where $g_{picard} $ is the solution operator of the linear problem \eqref{eq:pi_grad_div} at step $k$ (the stability bounds above are sufficient to show $g_{picard}$ is well defined), and hence the residuals of the iteration can be measured with the difference in successive iterates.  Of course, since there is noise in the data, the limit solution of this iteration will not be particularly meaningful.  Still, if one is going to run an iteration, then it is helpful to know what is the expected behavior of that iteration.

We first prove a convergence result for the residual, under a smallness condition on the PDE parameters, but for any $H>0$ or $\mu\ge 0$.  Then we prove a general convergence result with no smallness assumptions on PDE parameters, but with $H$ small enough and $\mu$ large enough.  Since these results follow very similarly to the convergence results for the case of no noise from \cite{LHRV23} (once two successive iterations are subtracted, the noise term drops), we omit the proofs here but move them to the appendix for interested readers.

\begin{thm}[Convergence of the residuals for small PDE parameters and noise but any $H$ or $\mu$] \label{rescon1}
Suppose $\alpha_{\delta}<\frac13$.  Then for sufficiently large $k$ and any $H$ and $\mu$, we have the bound
\[
\| \nabla (u_{k+1} - u_k ) \|  \le  C_{\delta}  \alpha   \| \nabla( u_k - u_{k-1}) \|.
\]
\end{thm}
\begin{remark}
For PDE parameters and noise sufficiently small so that $C_{\delta}  \alpha = 2\sqrt{1+ (1+\delta)^2 }\alpha <1$, the iteration will converge (asymptotically) linearly.  
\end{remark}

We now state a result for residual convergence for any data and noise size, but with $H$ small enough and $\mu$ large enough.  Recall the definition of $\hat \lambda$ in \eqref{lambdahat}.

\begin{thm}[Convergence of the residuals for general PDE parameters and noise, provided $H$ is small enough and $\mu$ is large enough]\label{rescon2}
Suppose $\mu$ is chosen so that $\mu\ge \frac{\nu}{2C_I^2 H^2}$ and that $H$ is small enough so that $\rho = 2C_I C_{\delta}^2 \alpha^2 H <1$.  Then
\begin{align*}
\frac{\nu}{4} \| \nabla (u_{k+1} - u_k ) \|^2 + \hat \lambda \| u_{k+1} - u_k \|^2 
\le
2C_I C_{\delta}^2 \alpha^2 H \left( \frac{\nu}{4} \| \nabla( u_k - u_{k-1}) \|^2 +  \hat \lambda \| u_k - u_{k-1} \|^2 \right).
\end{align*}
\end{thm}

\begin{remark}
This result is similar to the case of convergence when no noise is present, and the noise effect is seen in the $C_{\delta}$ term.  This result implies that smaller $H$ and larger $\mu$ will help the iteration converge to a limit solution.  However, the limit solution will likely be inaccurate if the noise level is high.
\end{remark}

\subsection{Error analysis of the CDA-Picard with noisy data iteration}

We now consider error analysis of the CDA Picard iteration with noisy data.  Even though we have proved above that the iteration itself will converge (under appropriate assumptions on the data, noise, and CDA parameters), it is not expected that the limit solution will be accurate.  Hence we consider here the error itself, and find two results, which both show the error is bounded by the size of the noise.  We denote the error at step $k$ by $e_k:= u_k - u$, where $u$ is the NSE solution from which the measurement data were obtained.

Our first result shows that with enough measurement points and sufficiently small PDE parameters, an $L^2$ error estimate can be proven that depends on the size of the error but is independent of $\mu$.

\begin{thm} [Error at Step k] \label{thmerr1}
Let $0<r<1$.  Pick parameters $\mu$ and $H$ such that $ \frac{4\nu\alpha^4}{r^2} \le \mu \le \frac{\nu}{C_I^2 H^2}$ and $H\le \frac{r}{2C_I \alpha^2}$.  Then for $k$ sufficiently large, the error in step $k$ of the iteration satisfies
\[
\| e_k \| \le \sqrt{ \frac{4}{1-r} } \| I_H \epsilon \|.
\]

\end{thm}

\begin{proof}
We begin by subtracting the scheme from the NSE (the solution $u$ from which the measurements were taken), and then testing with $e_{k+1}$, to obtain
\begin{equation}
\nu \| \nabla e_{k+1} \|^2 + \mu \| I_H e_{k+1} \|^2 = - b(e_k,u,e_{k+1}) - \mu(I_H \epsilon,I_H e_{k+1}). \label{n0}
\end{equation}
Bounding the last right hand side term with Cauchy-Schwarz and Young, we get that
\[
- \mu(I_H \epsilon,I_H e_{k+1}) \le \frac{\mu}{2} \| I_H \epsilon \|^2 + \frac{\mu}{2} \| I_H e_{k+1} \|^2,
\] 
and so
\begin{equation}
\nu \| \nabla e_{k+1} \|^2 + \frac{\mu}{2} \| I_H e_{k+1} \|^2 = - b(e_k,u,e_{k+1}) +  \frac{\mu}{2} \| I_H \epsilon \|^2. \label{n1}
\end{equation}

Proceeding similar to above, we can lower bound the left side via
\begin{align*}
\nu \| \nabla e_{k+1} \|^2 + \frac{\mu}{2} \| I_H e_{k+1} \|^2 & = \nu \| \nabla e_{k+1} \|^2 + \frac{\mu}{2} \| I_H e_{k+1} \|^2 \\
& = \frac{\nu}{2} \| \nabla e_{k+1} \|^2 + \frac{\nu}{2} \| \nabla e_{k+1} \|^2 + \frac{\mu}{2} \| I_H e_{k+1} \|^2 \\
& \ge \frac{\nu}{2} \| \nabla e_{k+1} \|^2 + \frac{\nu}{2 C_I^2 H^2} \| e_{k+1} - I_H e_{k+1} \|^2 + \frac{\mu}{2} \| I_H e_{k+1} \|^2,
\end{align*}
thanks to the property \eqref{interpolationi}.  Setting 
\[
\lambda = \min \bigg\{ \frac{\mu}{4}, \frac{\nu}{4C_I^2 H^2} \bigg\},
\]
we use the triangle inequality to get the bound
\begin{align*}
\nu \| \nabla e_{k+1} \|^2 + \frac{\mu}{2} \| I_H e_{k+1} \|^2 & \ge \frac{\nu}{2} \| \nabla e_{k+1} \|^2 + \lambda \| e_{k+1} \|^2.
\end{align*}
Using this in \eqref{n1} and then bounding the right hand side of \eqref{n1} using \eqref{bbound1} to get
\begin{align*}
\frac{\nu}{2} \| \nabla e_{k+1} \|^2 + \lambda \| e_{k+1} \|^2 
& \le  - b(e_k,u,e_{k+1})+ \frac{\mu}{2} \| I_H \epsilon \|^2. \\
& \le M \| e_k \|^{1/2} \| \nabla e_k \|^{1/2} \| \nabla u \| \| \nabla e_{k+1} \| + \frac{\mu}{2} \| I_H \epsilon \|^2 \\
& \le \nu \alpha \| e_k \|^{1/2} \| \nabla e_k \|^{1/2}  \| \nabla e_{k+1} \| + \frac{\mu}{2} \| I_H \epsilon \|^2 \\
& \le \frac{\nu}{4}\| \nabla e_{k+1} \|^2 + \nu \alpha^2 \| e_k \| \| \nabla e_k \|  + \frac{\mu}{2} \| I_H \epsilon \|^2.
\end{align*}
Applying Cauchy-Schwarz on the middle right hand side term and reducing yields for any $B>0$ that
\begin{align*}
\frac{\nu}{4} \| \nabla e_{k+1} \|^2 + \lambda \| e_{k+1} \|^2 
& \le \nu \frac{B}{2} \| \nabla e_k \|^2 + \frac{\nu \alpha^4}{2B}  \| e_k \|^2 + \frac{\mu}{2} \| I_H \epsilon \|^2\\
& \le r \left( \nu \frac{B}{2r} \| \nabla e_k \|^2 + \frac{\nu \alpha^4}{2B r}  \| e_k \|^2 \right) + \frac{\mu}{2} \| I_H \epsilon \|^2,
\end{align*}
where $0<r<1$.  Setting $B = \frac{r}{2}$ and noting the assumptions on $\mu$ and $H$ from the theorem statement imply
that 
\[
 \frac{\nu \alpha^4}{r^2} \le  \lambda =  \min \bigg\{ \frac{\mu}{4}, \frac{\nu}{4C_I^2 H^2} \bigg\},
\]
and thus since $r<1$
\begin{align*}
\frac{\nu}{4} \| \nabla e_{k+1} \|^2 + \lambda \| e_{k+1} \|^2  & \le r \left( \frac{\nu}{4} \| \nabla e_{k} \|^2 + \lambda \| e_{k} \|^2  \right) + \frac{\mu}{2} \| I_H \epsilon \|^2 \\
& \le  r^2 \left( \frac{\nu}{4} \| \nabla e_{k-1} \|^2 + \lambda \| e_{k-1} \|^2  \right) + (1+r) \frac{\mu}{2} \| I_H \epsilon \|^2 \\
& \le r^{k+1} \left( \frac{\nu}{4} \| \nabla e_0 \|^2 + \lambda \| e_0 \|^2  \right) + \frac{\mu}{2 (1-r) } \| I_H \epsilon \|^2. 
\end{align*}
Thus for $k$ sufficiently large such that
$$
r^{k+1} \left( \frac{\nu}{4} \| \nabla e_0 \|^2 + \lambda \| e_0 \|^2  \right)\le \frac{\mu}{2 (1-r) } \| I_H \epsilon \|^2,
$$
and using the assumption that $\mu\le \frac{\nu}{C_I^2 H^2}$,
\[
\frac{\mu}{4} \| e_{k+1} \|^2 = \min \bigg\{ \frac{\mu}{4}, \frac{\nu}{4C_I^2 H^2} \bigg\} \| e_k \|^2 = \lambda \| e_k \|^2 \le \frac{\mu}{1-r} \| I_H \epsilon \|^2,
\]
which reduces to 
\[
\| e_k \| \le \sqrt{ \frac{4}{1-r} } \| I_H \epsilon \|,
 \]
 finishing the proof.
 \end{proof}

The theorem above proved an error bound under (essentially) a smallness condition on the PDE parameters.  The theorem below proves an error bound for any data, and does not require the NSE solution be unique.  However, we do require a single NSE solution is used for measurement data.  We denote
\begin{equation}\label{gama}
\bar \lambda:=\min\left\{\frac{\nu}{ C_I^{2}H^{2}},\frac{\mu}{2}\right\}
\end{equation}
in the theorem.

\begin{thm} [Error at Step k for general PDE parameters] \label{thmerr2}
Let $u$ be a steady NSE solution and $I_H(u)$ be its interpolant on $X_H$.  Pick parameters $\mu$ and $H$ such that $\bar \lambda > 2$, $H\le \frac{\nu^{1/2}}{\sqrt{2} K_1 C_I}$, and $\mu \ge 4K_1^2$.  Then it holds that
\begin{equation}\label{cota1}
\| e_k\|^2 \le  \left(\frac{2}{\bar \lambda}\right)^{k}\| e_0 \|^2+\frac{2 \mu }{\bar \lambda-2} \|I_H\epsilon\|^2.
\end{equation}
Moreover, if $\mu\le \frac{\nu}{C_I^2H^2}$ (equivalently $H\le \frac{\sqrt{2}\nu^{1/2}}{ \mu^{1/2}C_I}$)  then
\begin{equation}\label{cota2}
\| e_k\|^2 \le  \left(\frac{2}{\bar \lambda}\right)^{k}\| e_0 \|^2+\frac{4 \mu }{\mu-4} \|I_H\epsilon\|^2.
\end{equation}
\end{thm}
\begin{remark}
Hence we observe linear convergence of the error in the $L^2$ norm up to $\left( \frac{2 \mu }{\bar \lambda-2}\right)^{1/2} \|I_H\epsilon\|$, provided $\mu$ is large enough and $H$ is small enough to satisfy the assumptions.  Compared to Theorem \ref{thmerr1}, the improvement of this theorem is that there is no restriction on PDE parameters $Re$, $f$, and $M$. The tradeoff is that the term representing the noise is scaled by $\mu$, which depends on $\| \nabla u \|_{L^{\infty}}^2$, see \eqref{cota1}.
However, assuming $\mu\le \frac{\nu}{C_I^2H^2}$,  as in Theorem \ref{thmerr1}, or
equivalently $H$  small enough, $H\le \frac{\sqrt{2}\nu^{1/2}}{ \mu^{1/2}C_I}$, we get \eqref{cota2} in which the constant $\frac{4 \mu }{\mu-4}$
is bounded as $\mu\rightarrow \infty$.
\end{remark}

\begin{proof}
We begin the proof the same was as for the proof of Theorem \ref{thmerr1} and thus begin with the error equation
\[
\nu \|\nabla \be_{k+1}\|^2+\mu \|I_H \be_{k+1}\|^2=-b(\be_k,\bu,\be_{k+1})-\mu (I_H\epsilon,I_H\be_{k+1}).
\]
Applying Young's inequality on the last right hand side term (just as in Theorem \ref{thmerr1}, but now estimating the nonlinear
term using H\"older's inequality with $L^2-L^{\infty}-L^2$) gives
\begin{eqnarray*}
\nu \|\nabla \be_{k+1}\|^2+\frac{\mu}{2} \|I_H \be_{k+1}\|^2&\le& \|\be_k\|\|\nabla \bu\|_{\infty}\|\be_{k+1}\|
+\frac{\mu}{2} \|I_H \epsilon\|^2\\
&\le& \frac{1}{2}\|\nabla \bu\|_{\infty}^2\|\be_{k+1}\|^2+\frac{1}{2}\|\be_{k}\|^2+\frac{\mu}{2} \|I_H \epsilon\|^2\\
&\le& \frac{1}{2}K_1^2\|\be_{k+1}\|^2+\frac{1}{2}\|\be_{k}\|^2+\frac{\mu}{2} \|I_H \epsilon\|^2.
\end{eqnarray*}
with the last step thanks to Young's inequality.  Denoting $L:=\frac{1}{2}K_1^2$ and  arguing as in \cite{GNT18} we can write
\begin{equation}
L\|\be_{k+1}\|^2\le 2L \|I_H\be_{k+1}\|^2+2L\|(I-I_H)\be_{k+1}\|^2. \label{tri}
\end{equation}
Since $\mu\ge 4K_1^2$ is assumed, we have that $\mu\ge 8L$ and so $\mu/2-2L \ge \mu/4$ and thus
\begin{eqnarray*}
\nu \|\nabla \be_{k+1}\|^2+\frac{\mu}{4}\|I_H \be_{k+1}\|^2-2L\|(I-I_H)\be_{k+1}\|^2\le \frac{1}{2}\|\be_{k}\|^2
+\frac{\mu}{2} \|I_H \epsilon\|^2.
\end{eqnarray*}
Now using \eqref{interpolationi} and the assumption on $H$ we get that
\[
\nu \|\nabla \be_{k+1}\|^2-2L\|(I-I_H)\be_{k+1}\|^2\ge \nu \|\nabla \be_{k+1}\|^2-2LC_I^2H^2\|\nabla \be_{k+1}\|^2
\ge \frac{\nu}{2}\|\nabla \be_{k+1}\|^2,
\]
which combines with the previous inequality to yield
\begin{eqnarray*}
{\nu} \|\nabla \be_{k+1}\|^2+\frac{\mu }{2}\|I_H \be_{k+1}\|^2\le \|\be_{k}\|^2+{\mu}\|I_H \epsilon\|^2.
\end{eqnarray*}
Applying \eqref{interpolationi} again to lower bound the left hand side (similar to in the proofs above) gives us that
\begin{eqnarray*}
\nu C_I^{-2}H^{-2} \|(I-I_H)\be_{k+1}\|^2+\frac{\mu }{2}\|I_H \be_{k+1}\|^2&\le&{\nu} \|\nabla \be_{k+1}\|^2+\frac{\mu }{2}\|I_H \be_{k+1}\|^2\\
&\le& \|\be_{k}\|^2+\mu \|I_H \epsilon\|^2.
\end{eqnarray*}

Next, using the definition of  $\bar \lambda$ and  \eqref{tri} (after dividing both sides by $2L$) yields
$$
\frac{\bar \lambda}{2}\|\be_{k+1}\|^2\le \|\be_k\|^2+\mu \|I_H \epsilon\|^2,
$$
so that
$$
\|\bu_{k+1}-\bu\|^2\le \frac{{2}}{{\bar \lambda}} \|\bu_k-\bu\|^2+\frac{{2\mu}}{{\bar \lambda}} \|I_H \epsilon\|^2.
$$
Denoting $r=\frac{{2}}{{\bar \lambda}} <1,$ (by assumption of the theorem) we can step the iteration backward and obtain the bound
\begin{eqnarray*}
\|\bu_{k+1}-\bu\|^2&\le& r^{k+1}\|\bu_0-\bu\|^2+(1+r+\ldots+r^k)\mu r\|I_H\epsilon\|^2\\
&\le& r^{k+1}\|\bu_0-\bu\|^2+\frac{\mu r}{1-r} \|I_H\epsilon\|^2 \\
& \le & \left(\frac{2}{\bar \lambda}\right)^{k+1}\|\bu_0-\bu\|^2+\frac{2 \mu }{\bar\lambda-2} \|I_H\epsilon\|^2. 
\end{eqnarray*}
This proves \eqref{cota1}. In case $\mu\le \frac{\nu}{C_I^2H^2}$ then $\bar\lambda=\mu/2$ which gives \eqref{cota2}.
\end{proof}

\section{Numerical Experiments}

In this section we illustrate the above theory and the effectiveness of the proposed method on the 2D driven cavity at $Re$=3000 and 10000, and the 3D driven cavity at $Re$=200 and 1000.  CDA-Picard $L^2$ error is shown in all tests to converge to approximately the level of the signal to noise ratio  while its $L^2$ residual (the difference between iterates) converges linearly to 0.  In addition, we propose and test the strategy of using CDA-Picard with noisy data to create an initial guess for usual Newton.   This is shown to be a very effective solver for these test problems.

\subsection{2D driven cavity with varying signal to noise ratio}

\begin{figure}[ht]
\center
Re=3,000 \hspace{1in}  Re=10,000 \\
\includegraphics[width = .3\textwidth, height=.28\textwidth,viewport=115 45 465 390, clip]{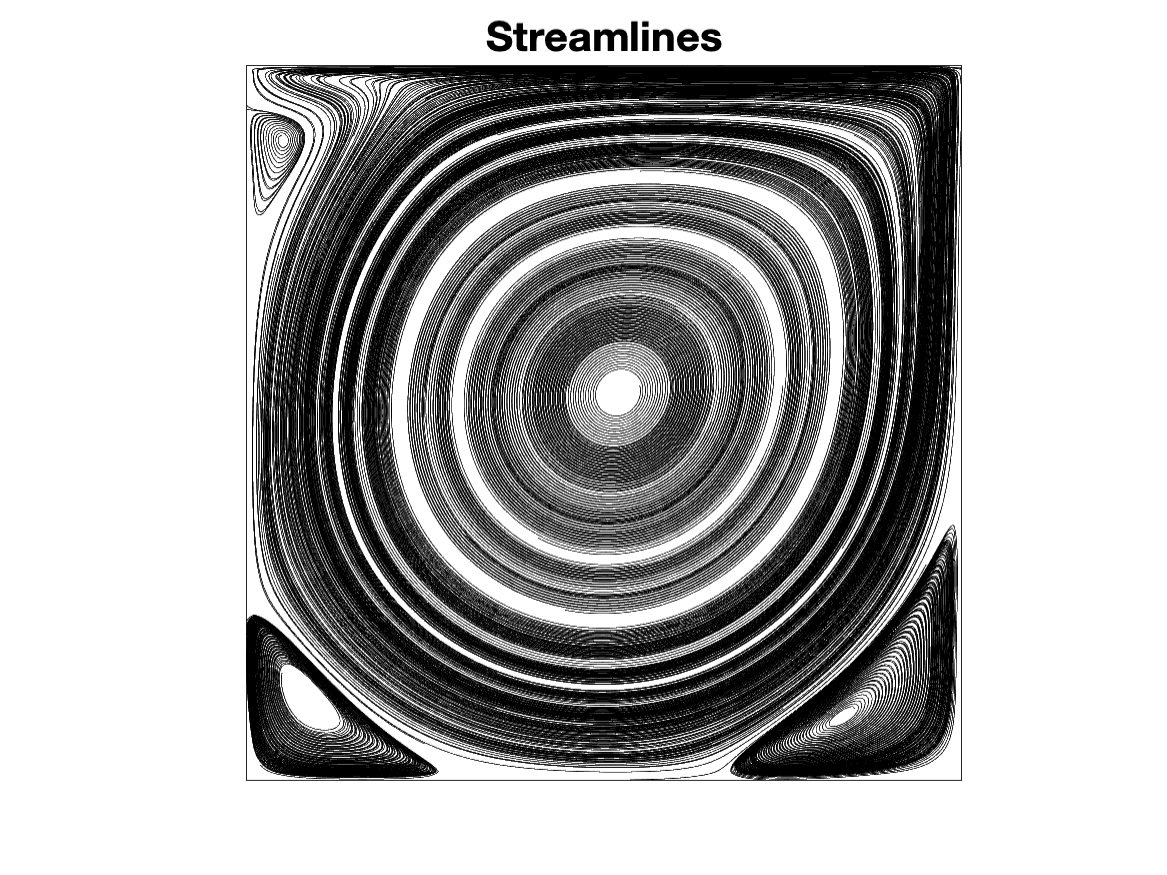}  
\includegraphics[width = .3\textwidth, height=.28\textwidth,viewport=115 45 465 390, clip]{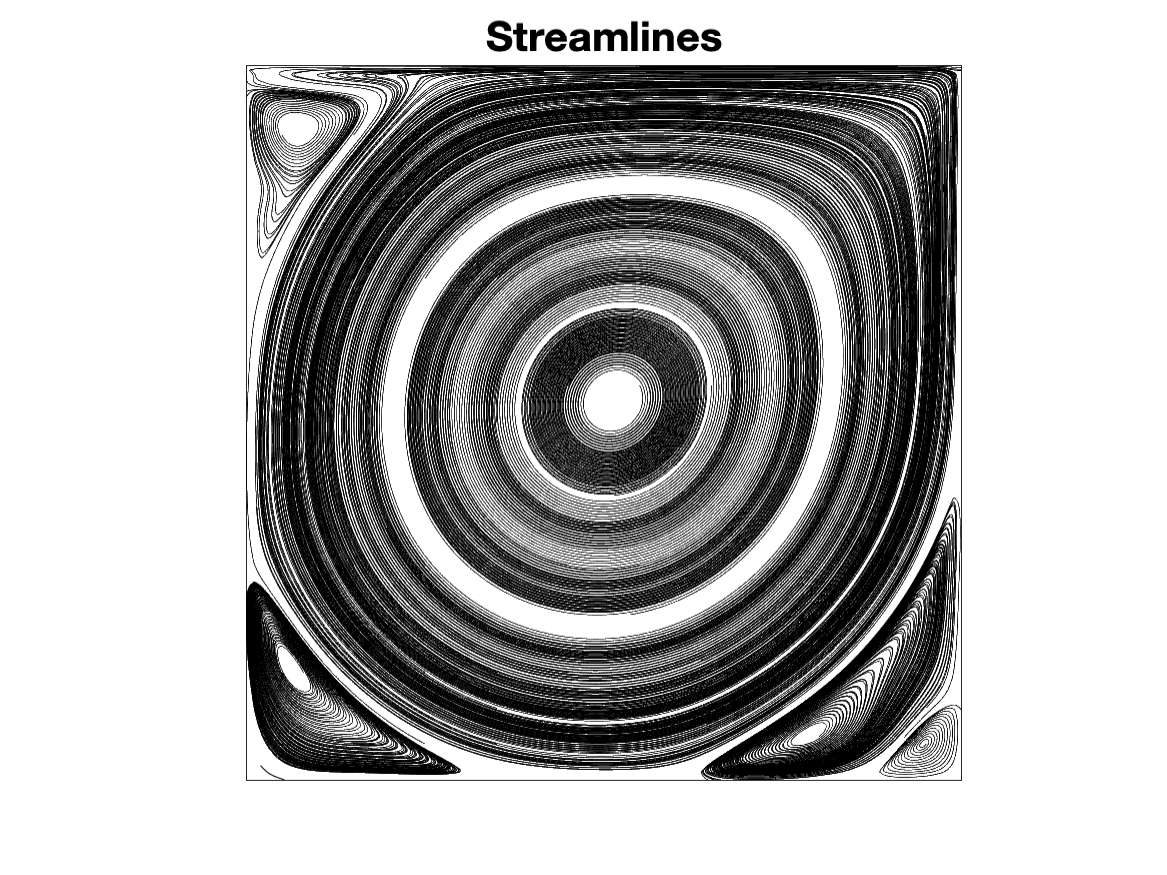}  
\caption{\label{nseplot1} The plot above shows streamlines of the solution of the 2D driven cavity problem at $Re=$3,000 (left) and 10,000 (right).}
\end{figure}

For our first test we consider the CDA-Picard method for the NSE applied to the benchmark 2D driven cavity problem on $\Omega=(0,1)^2$.  Boundary conditions are zero on the sides and bottom, and on the top we enforce $u=\langle 1,0 \rangle^T$ which models a lid moving horizontally with unit velocity.  No external forcing is used in this problem ($f=0$) and the viscosity $\nu$ will be chosen as the inverse of the Reynolds number, and we test with $Re$=3000, and 10000 (the first of which is believed to satisfy the smallness condition $\alpha<1$ and the second of which does not \cite{PRX19}).  

Computations are performed using $(P_2,P_1^{disc})$ Scott-Vogelius elements on a barycenter refinement (also named the Alfeld split in the Guzman-Neilan vernacular) of a uniform $\frac{1}{64}$ triangulation that has southwest-northeast diagonals.  It is known from \cite{arnold:qin:scott:vogelius:2D} that this element choice is inf-sup stable on this type of structured mesh.  This discretization provides 98,818 total velocity degrees of freedom (dof) and 73,728 pressure dof.  For all of these tests, an initial iterate of $u_0=0$ is chosen.  Plots of the NSE solutions found on this discretization for these problems are shown in figure \ref{nseplot1}, and they agree well with existing literature \cite{bruneau:cavity,ECG05}.

We obtain noisy data measurements by first computing the NSE solution directly on the given discretization for each $Re$ (using the AA-Picard method from \cite{PRX19}).  Measurement points are taken to be the vertices of the mesh that are closest to the midpoints of an $N\times N$ uniform square grid of $\Omega$, where $N$=10 for $Re$=3000 and $N$=20 for $Re$=10000.  Signal to noise ratios (snr) were chosen to be 0.001, 0.01 and 0.05, and for each snr, random numbers were generated at each measurement point that come from a normal distribution with mean 0 and standard deviation $snr \times u_{max}$ ($u_{max}$ is the max velocity of the true solution, which is 1 for this test problem).  These noisy measurement data points were then used with CDA-Picard.  For $Re$=3000, 100 total measurement points were used and for $Re$=10000, 400 total measurement points were used (100 was not enough to yield convergence for $Re$=10000).

\begin{figure}[ht]
\center
\includegraphics[width = .32\textwidth, height=.3\textwidth,viewport=0 0 550 415, clip]{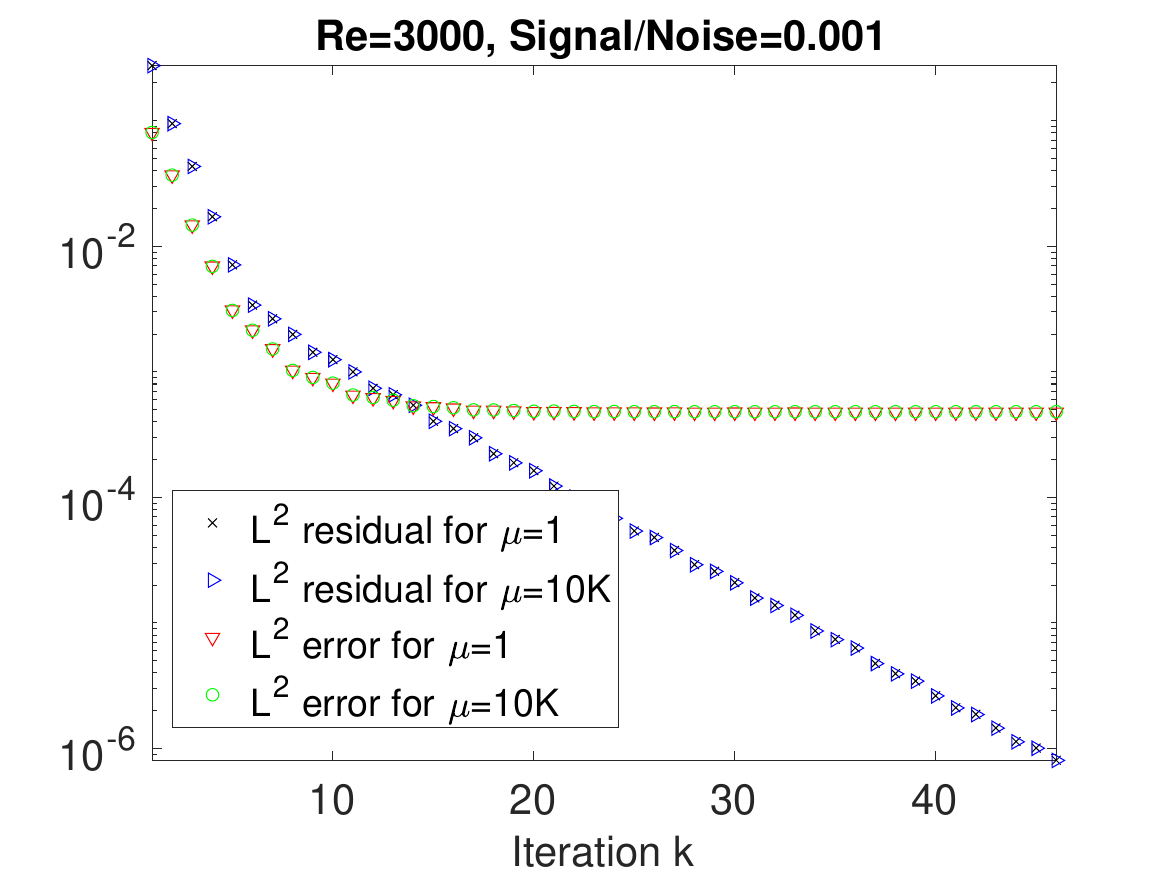}
\includegraphics[width = .32\textwidth, height=.3\textwidth,viewport=0 0 550 415, clip]{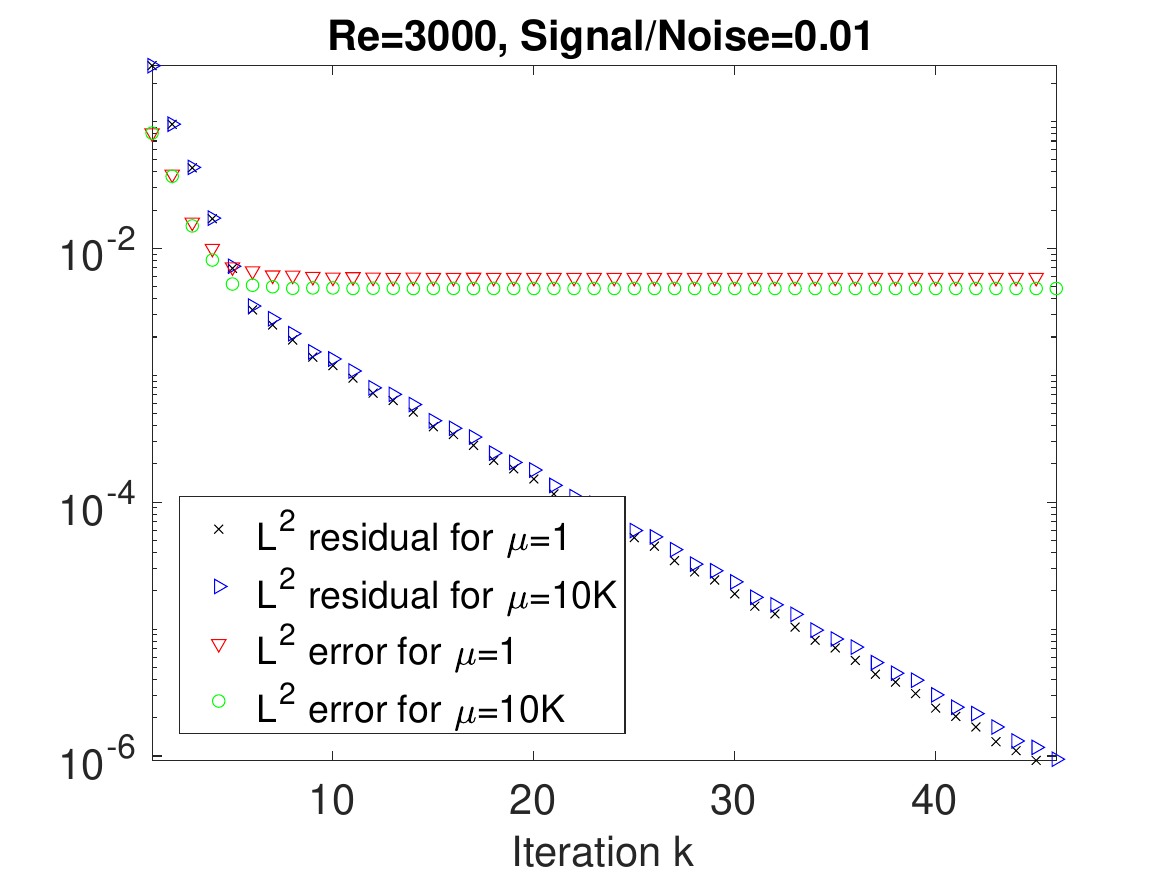}
\includegraphics[width = .32\textwidth, height=.3\textwidth,viewport=0 0 550 415, clip]{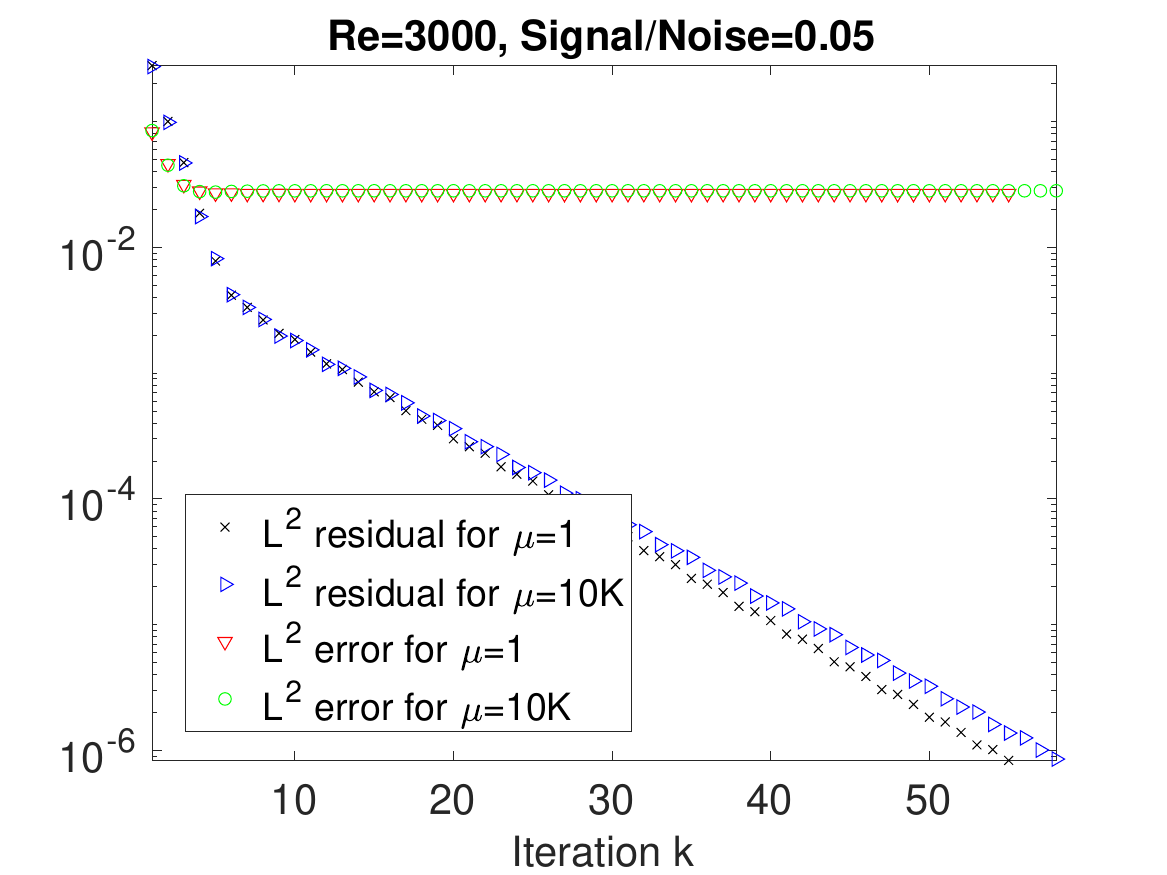}
\caption{\label{convplots3} Shown above are the error and residual plots for $Re$=3000 driven cavity tests with varying snr.}
\end{figure}

Convergence results for Re=3000 are shown in figure \ref{convplots3} for each choice of snr, and for both small and large nudging parameters ($\mu=1$ and $\mu=10,000$).  We observe that in all cases, the $L^2$ residuals converge approximately linearly, with little difference between large and small $\mu$ and with the total number of iterations to convergence increasing slightly as the snr increases.  The $L^2$ error plots, as expected, only converged up to approximately the level of the snr.

\begin{figure}[ht]
\center
\includegraphics[width = .32\textwidth, height=.3\textwidth,viewport=0 0 550 415, clip]{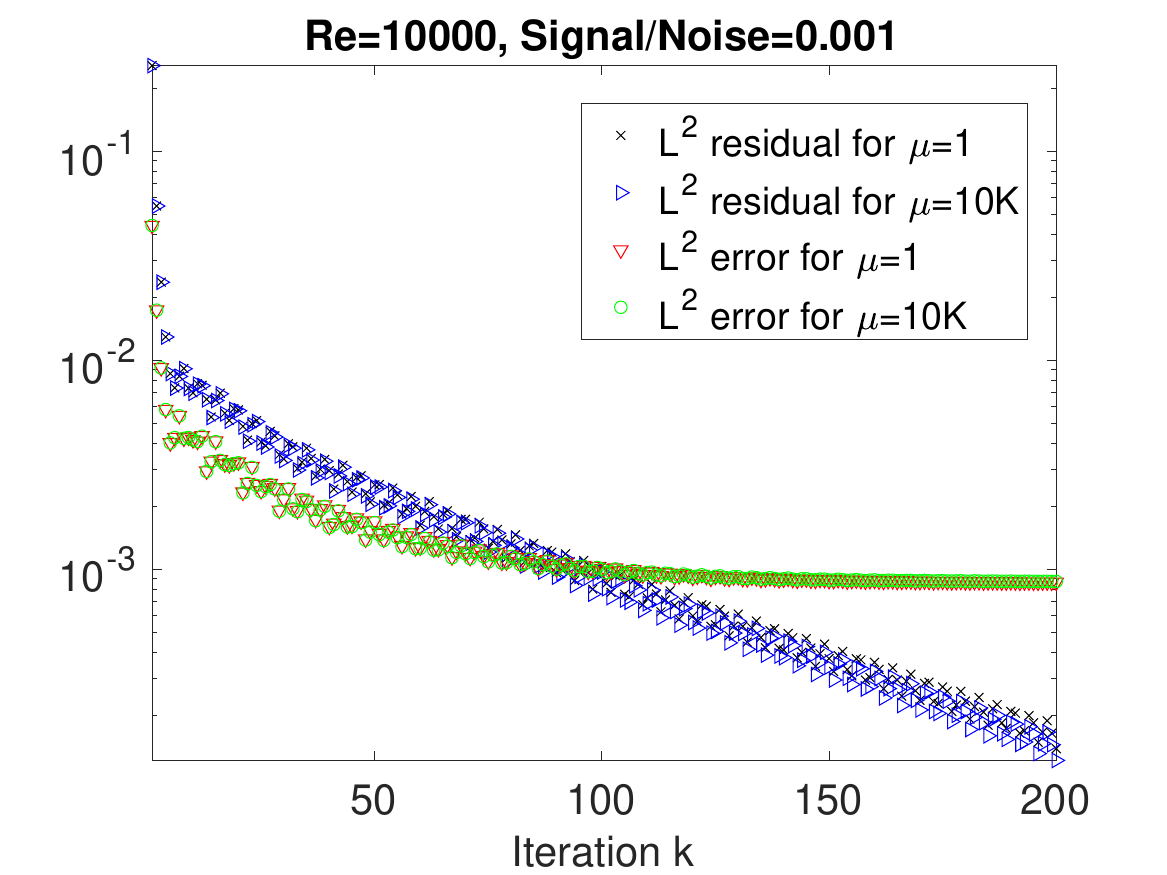}
\includegraphics[width = .32\textwidth, height=.3\textwidth,viewport=0 0 550 415, clip]{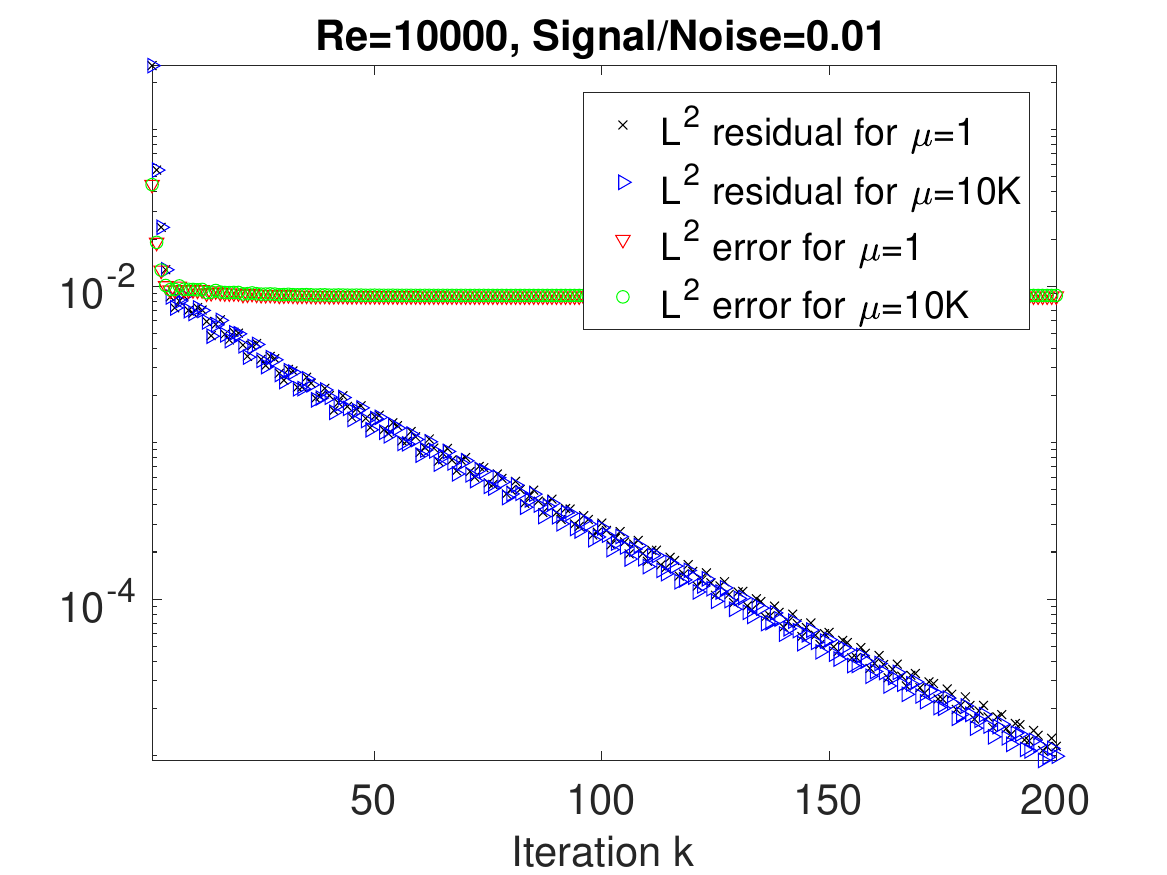}
\includegraphics[width = .32\textwidth, height=.3\textwidth,viewport=0 0 550 415, clip]{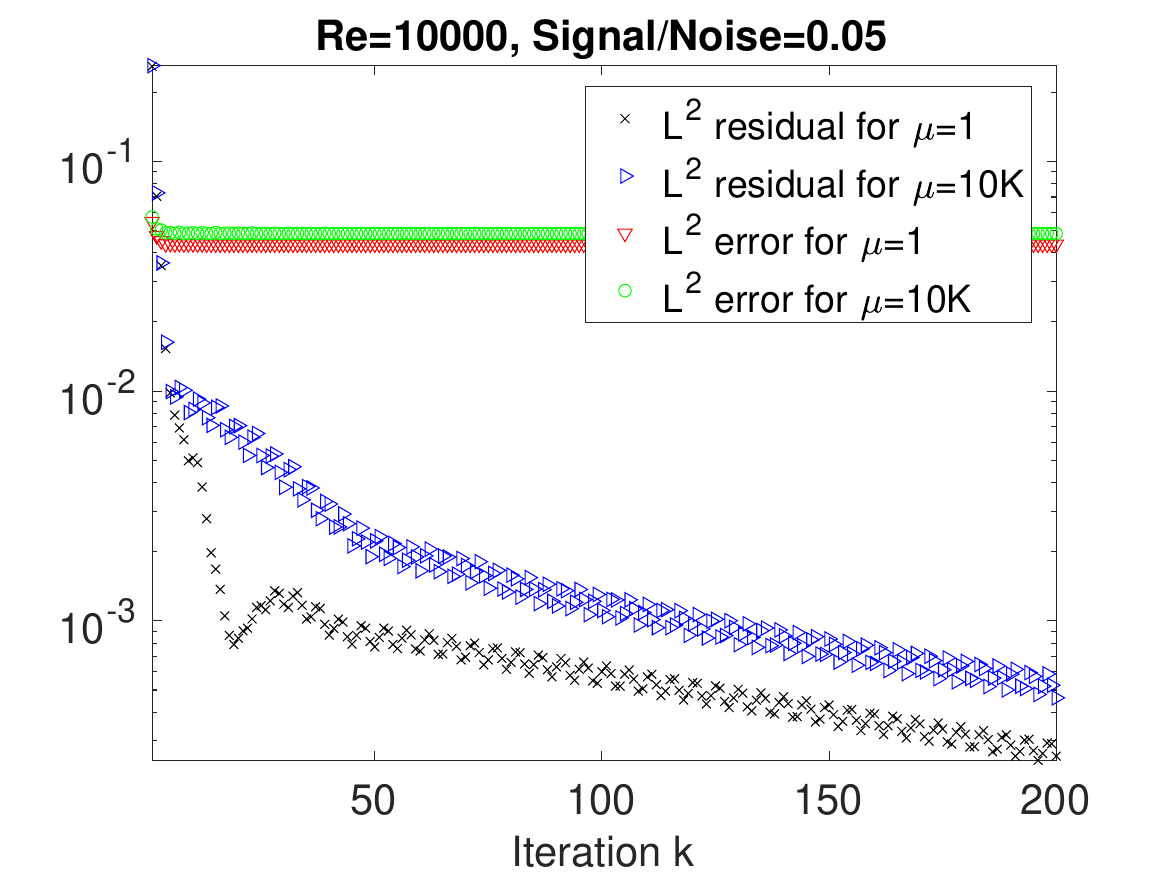}
\caption{\label{convplots10} Shown above are the error and residual plots for $Re$=10000 driven cavity tests with varying snr.}
\end{figure}

Results for Re=10000 are shown in figure \ref{convplots10}.  The $L^2$ error again converges approximately to the level of the snr.  The convergence of the $L^2$ residual shows somewhat different behavior than the Re=3000 case.  While as expected residual convergence is slower with higher Re, we observe better residual convergence for snr=0.01 than for 0.001, which is not expected.  We also observe for snr=0.05 that the larger nudging parameter gives somewhat worse behavior.  Of course, the interest in the $L^2$ residual is only useful in how it is associated with error.

\subsection{3D driven cavity with varying signal to noise ratio}

Our next test for CDA-Picard with noisy data uses the 3D lid-driven cavity benchmark problem.  This problem is the 3D analogue of the test problem above: the domain is the unit cube, there is no forcing, homogeneous Dirichlet boundary conditions are enforced on the walls and at the top of the box $u=\langle 1,0,0\rangle$ represents the moving lid.  The viscosity will be chosen as the inverse of the Reynolds number, and we will use $Re$=200 and 1000.  The tests will use $u_0=0$ for the initial condition for the iterations.

\begin{figure}[H]
			\centering
			\includegraphics[width = 0.95\textwidth, height=.27\textwidth,viewport=100 0 1100 300, clip]{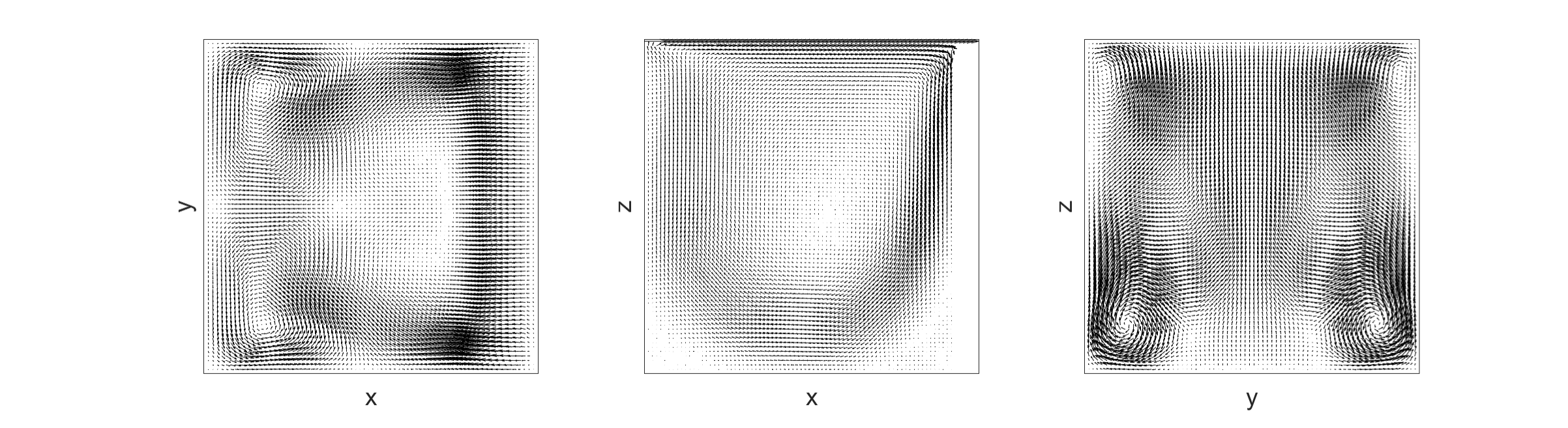}
			\caption{Shown above are the midsliceplane plots of solutions for the 3D driven cavity simulations at $Re$=1000.\label{fig:midsliceplanes}}
\end{figure}

The mesh is constructed using Chebychev points on [0,1] to first construct a $\mathcal{M}\times \mathcal{M}\times \mathcal{M}$ grid of rectangular boxes (we use $\mathcal{M}$=11 for $Re$=200 and $\mathcal{M}$=13 for $Re$=1000).  Each box is then split into 6 tetrahedra in the manner of figure \ref{tet}, and then each of these tetrahedra is split into 4 tetrahedra with a barycenter refinement (Alfeld split).  This mesh is equipped with 
$(P_3, P_2^{disc})$ Scott-Vogelius elements, and this provides for approximately 796K total dof for the $Re$=200 tests and 1.3 million total dof for the $Re$=1000 tests.  We note this velocity-pressure pair is known to be LBB stable on such a mesh construction from \cite{Z05_Alfed}.   Solution plots found with this discretization matched those from the literature \cite{WongBaker2002}, and we show midspliceplanes of the $Re$=1000 solution in figure \ref{fig:midsliceplanes}.

\begin{figure}[ht]
\center
\includegraphics[width = .32\textwidth, height=.3\textwidth,viewport=0 0 550 415, clip]{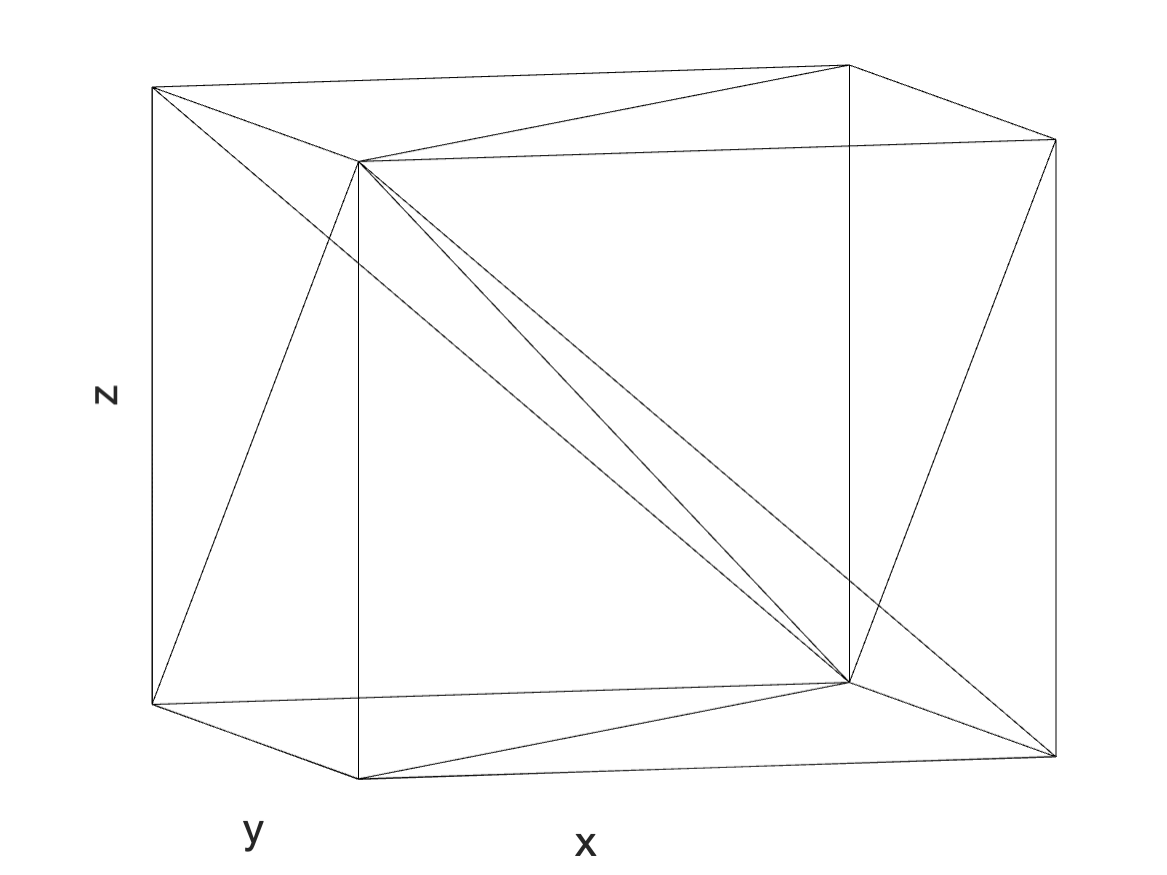}
\caption{\label{tet} Shown above is the method used to split a rectangular box into 6 tetrahedra.}
\end{figure}

The measurement data is constructed similar to the 2D case above.  Measurement points are taken to be the tetrahedra vertices closest to the centers of a uniform grid of  $N\times N\times N$ cubes.  Measurement data is taken to be the true solution (found using AA-Picard method of \cite{PRX19}) with noise added at each component from random numbers with normal distribution having mean 0 and standard deviation snr.  We again take snr to be 0.001, 0.01 and 0.05 for these tests.  These tests used $N$=4 (64 total measurement points) for $Re$=200 and $N$=9 (729 total measurement points) for $Re$=1000; for the latter case, $N$=8 was not sufficiently large to yield convergence (recall the theory above requires $H$ sufficiently small).  The nudging parameter $\mu$=1 was used, and we note that results from larger nudging parameters gave very similar results.

For $Re$=200, convergence of the $L^2$ residuals and $L^2$ errors for snr=0.001, 0.01 and 0.05 are shown in figure \ref{convplots4}, and for $Re$=1000 they are shown in figure \ref{convplots5}.  We observe approximately linear convergence of the $L^2$ residual down to the set tolerance of $10^{-8}$ in all cases, and also in all cases we observe the $L^2$ error to converge only to the level of the snr.  Note that the $Re$=1000 tests converge faster than the case of $Re=200$ but this is likely because many more measurement points were used (729 compared to 64).

\begin{figure}[ht]
\center
\includegraphics[width = .32\textwidth, height=.3\textwidth,viewport=0 0 550 415, clip]{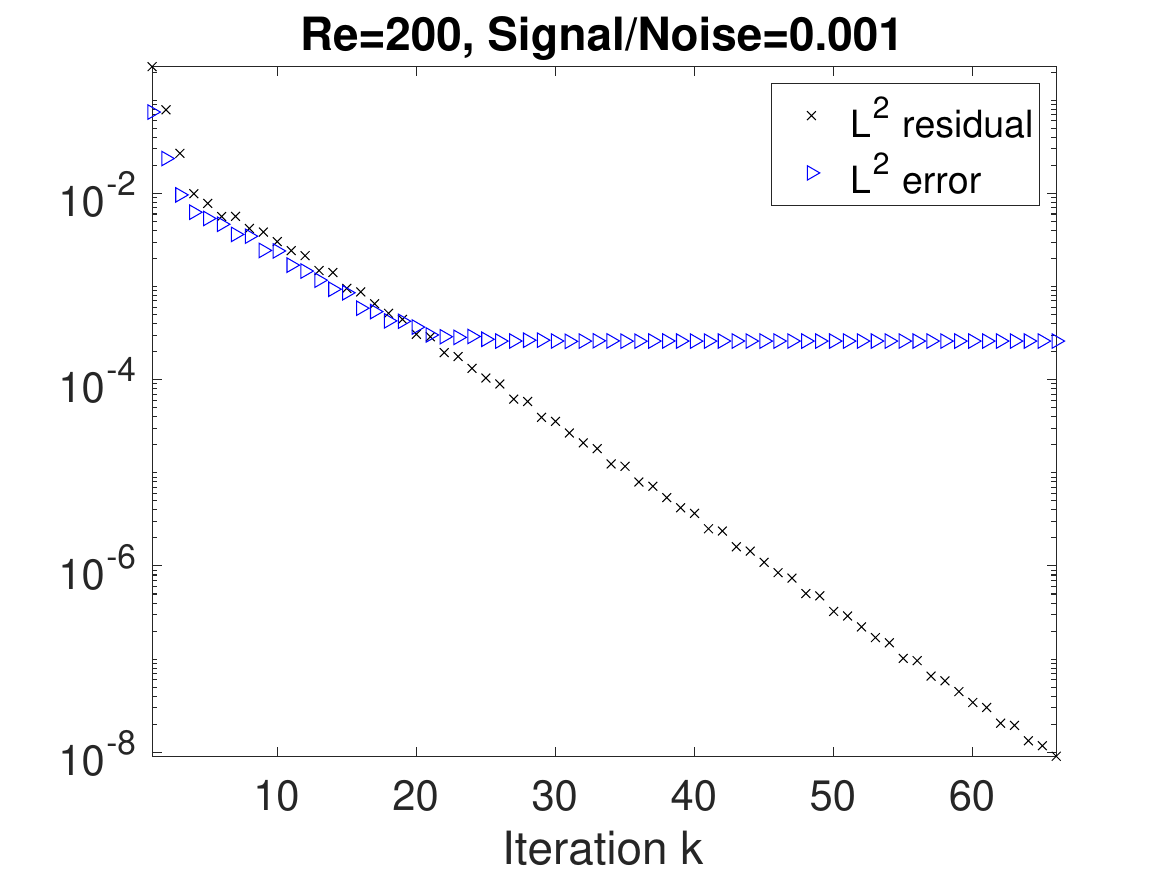}
\includegraphics[width = .32\textwidth, height=.3\textwidth,viewport=0 0 550 415, clip]{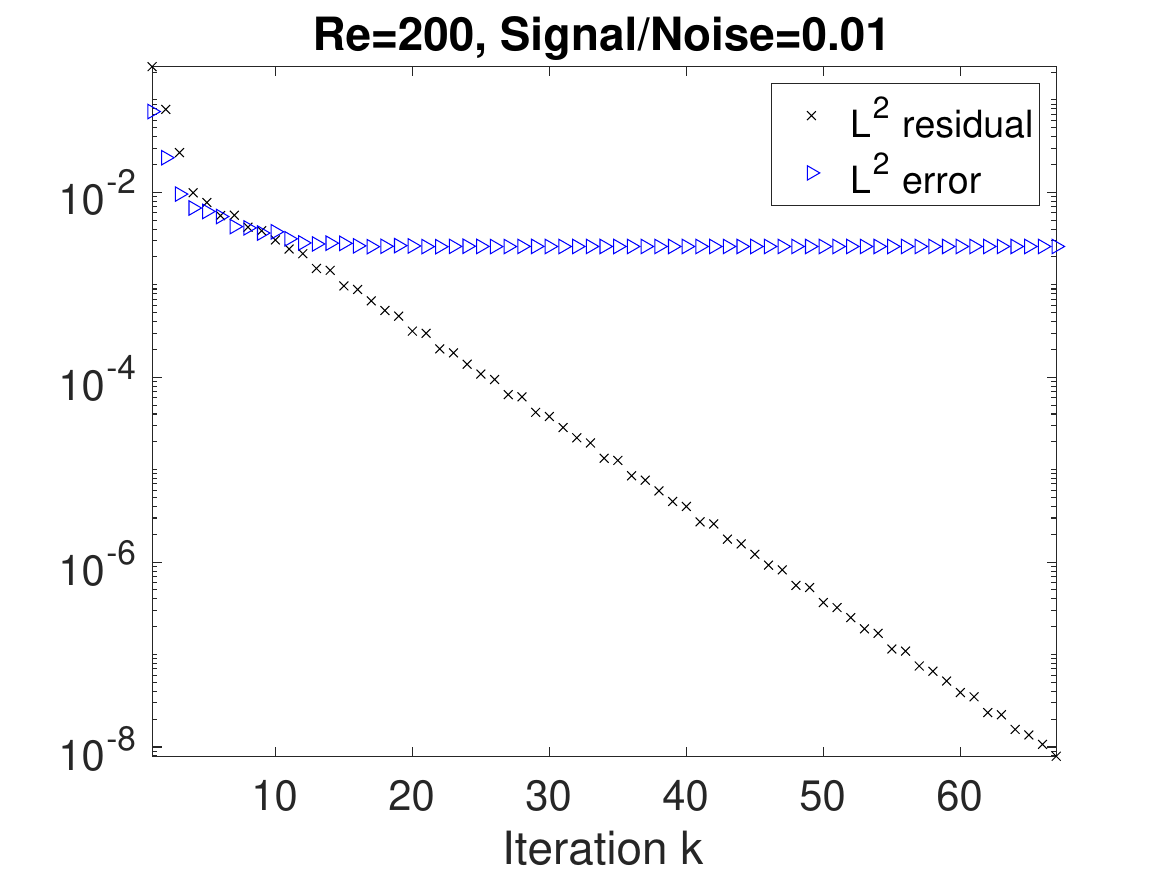}
\includegraphics[width = .32\textwidth, height=.3\textwidth,viewport=0 0 550 415, clip]{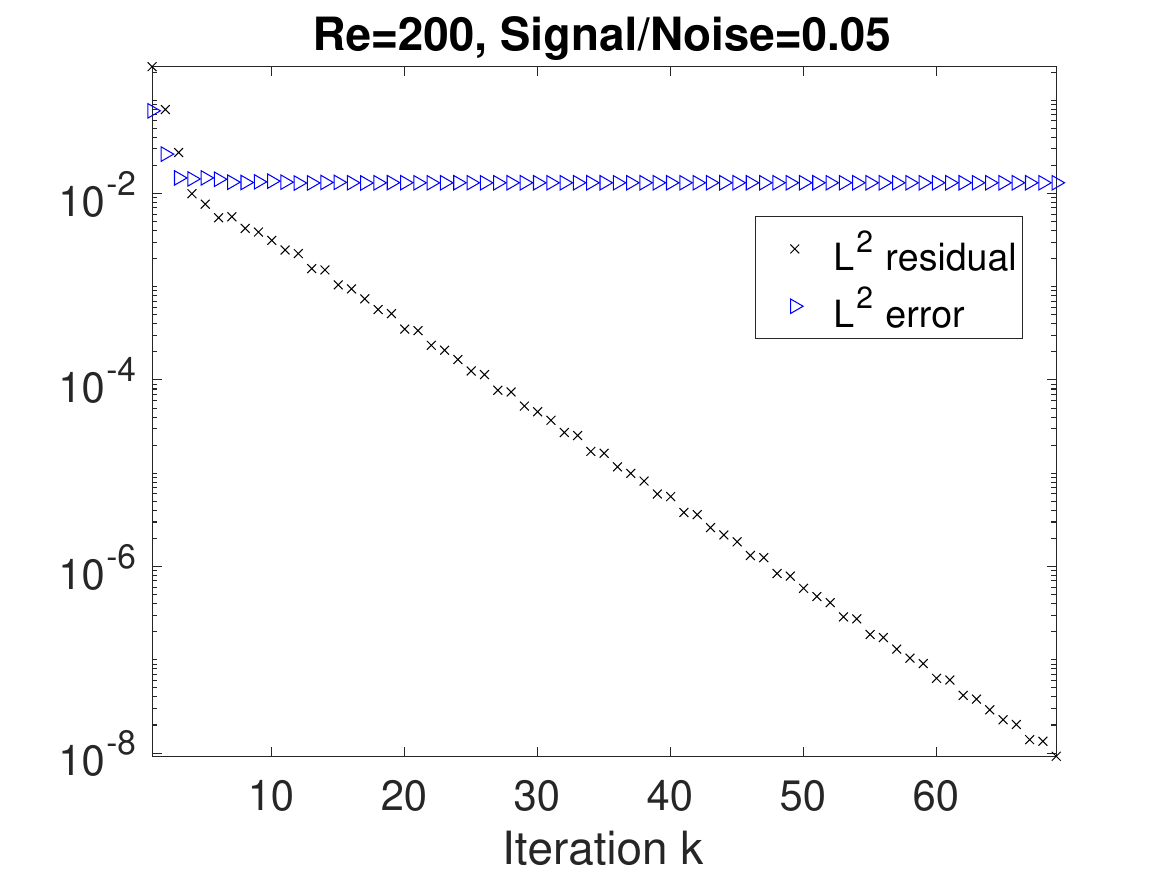}
\caption{\label{convplots4} Shown above are the error and residual plots for $Re$=200 3D driven cavity tests with varying snr.}
\end{figure}

\begin{figure}[ht]
\center
\includegraphics[width = .32\textwidth, height=.3\textwidth,viewport=0 0 550 415, clip]{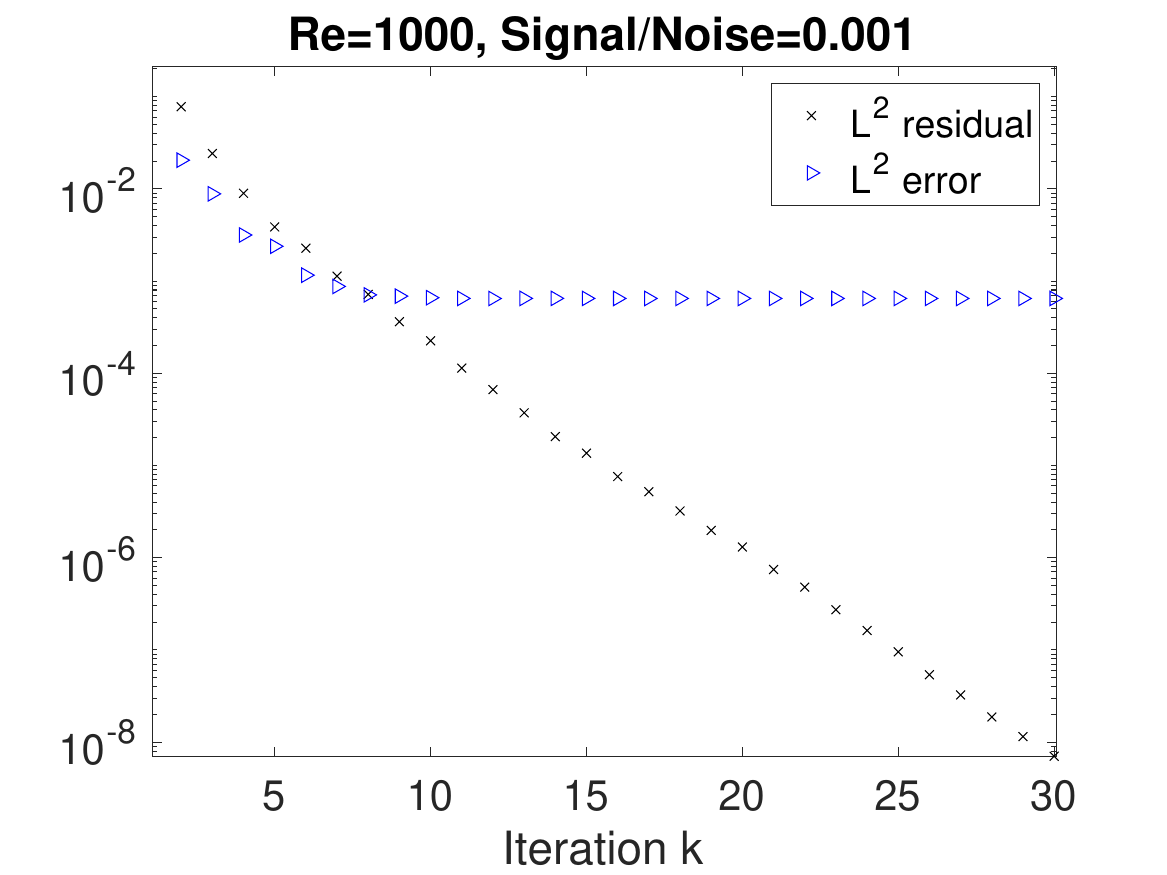}
\includegraphics[width = .32\textwidth, height=.3\textwidth,viewport=0 0 550 415, clip]{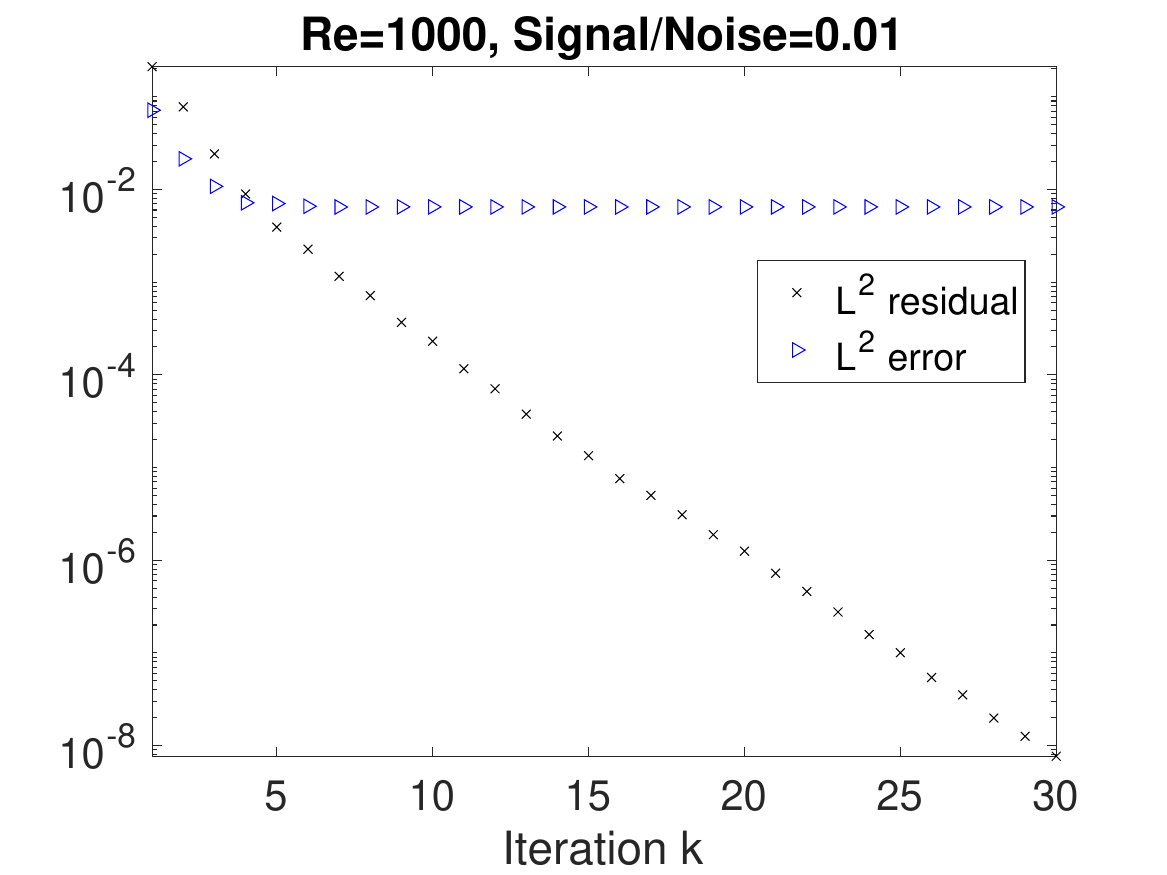}
\includegraphics[width = .32\textwidth, height=.3\textwidth,viewport=0 0 550 415, clip]{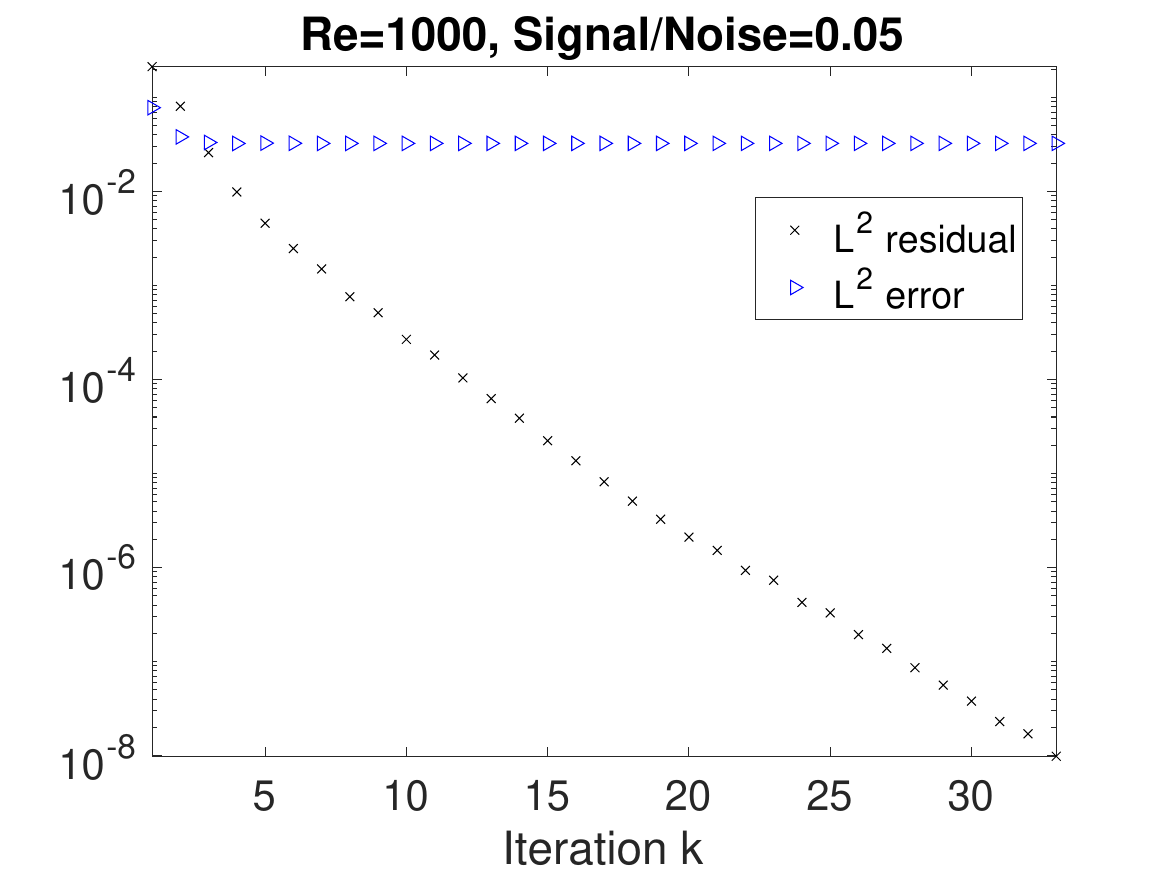}
\caption{\label{convplots5} Shown above are the error and residual plots for $Re$=1000 3D driven cavity tests with varying snr.}
\end{figure}

\subsection{Newton with CDA-Picard generated initial guess}

For our last test, we consider using CDA-Picard with noisy data for the purpose of generating an initial guess for the (usual) Newton iteration.  
From the tests and analysis above, we observe that CDA-Picard's $L^2$ residual will converge (provided enough measurement data), however the $L^2$ error will only converge up to approximately the level of the signal to noise ratio.   Also observed from the tests above, the $L^2$ residual and error are close to each other in value, until the $L^2$ error bottoms out near the snr while the $L^2$ residual continues to converge linearly.  
While noisy data will create a lower bound on the error for CDA-Picard iterates that is on the order of the signal to noise ratio, these iterates may be sufficiently close to the root that it can allow for Newton to converge if they are used as initial guesses.  Indeed, we find below that this strategy can be quite effective.

We now test this idea for the 2D driven cavity at $Re$=10000 and 3D driven cavity with $Re$=1000 problems tested above, with the same discretizations and parameters.  We note that both Newton and Picard iterations fail for these problems if an initial guess of $u_0=0$ is used.  We run CDA-Picard with noisy data (and varying snr) until the $L^2$ residual falls below $10^{-2}$, and then use the final iterate as the initial guess for the Newton iteration.  We refer to this method as CDA-Picard + Newton below.  We note that we also tested a $10^{-3}$ tolerance and got the same results, but $10^{-1}$ was not sufficiently accurate for Newton to converge.  We also  note that $\mu=1$ was the nudging parameter, but results using $\mu=10000$ gave very similar results.

Convergence results are shown in figure \ref{convplotsN} as $L^2$ error vs iteration count for CDA-Picard + Newton for varying snr, usual Newton and usual Picard.  The usual Newton, usual Picard, and CDA-Picard all use $u_0=0$ as the initial guess.  We observe that usual Picard fails to converge for both tests, although it remains stable (as expected \cite{GR86}).  The usual Newton iteration also fails and the iterates get very large; in the 2D case direct solvers are used and the iteration is able to continue but in 3D the linear solvers fail after the size of the iterate grows above $10^4$ (we use method of \cite{benzi,HR13}).  For CDA-Picard + Newton, however, convergence is achieved provided the snr is not too large.  For the 2D problem, convergence is achieved with snr=0.001 and 0.01, but the method fails to converge with snr=0.05.  For the 3D problem, convergence is achieved for snr=0.05 and 0.1 but not for 0.2.  Hence as expected, in both cases there is a level of noise which is too high to provide a good enough initial guess for Newton to converge.  Overall, these results show CDA-Picard + Newton is quite effective, provided the noise level is not too large.

\begin{figure}[ht]
\center
2D Re=10000 \hspace{1.7in} 3D Re=1000\\
\includegraphics[width = .4\textwidth, height=.33\textwidth,viewport=0 0 550 390, clip]{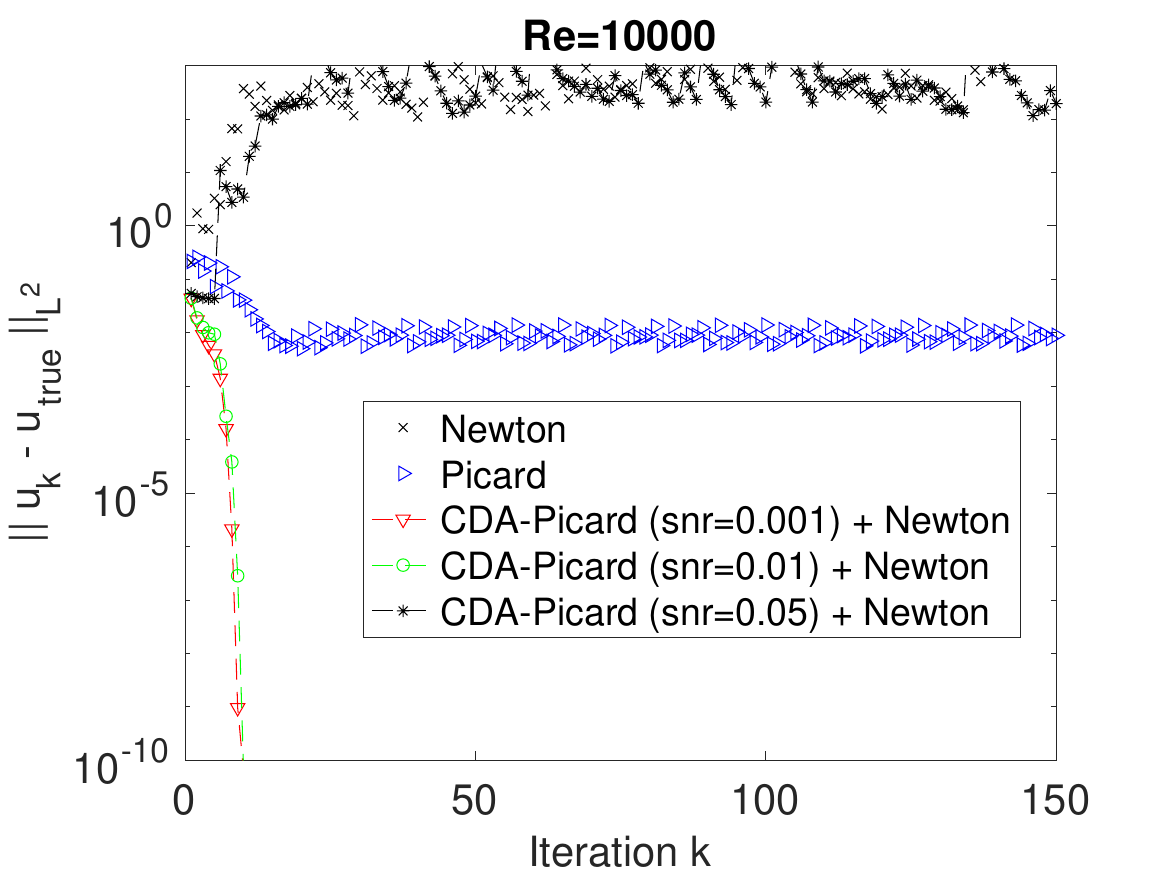}
\includegraphics[width = .4\textwidth, height=.33\textwidth,viewport=0 0 550 390, clip]{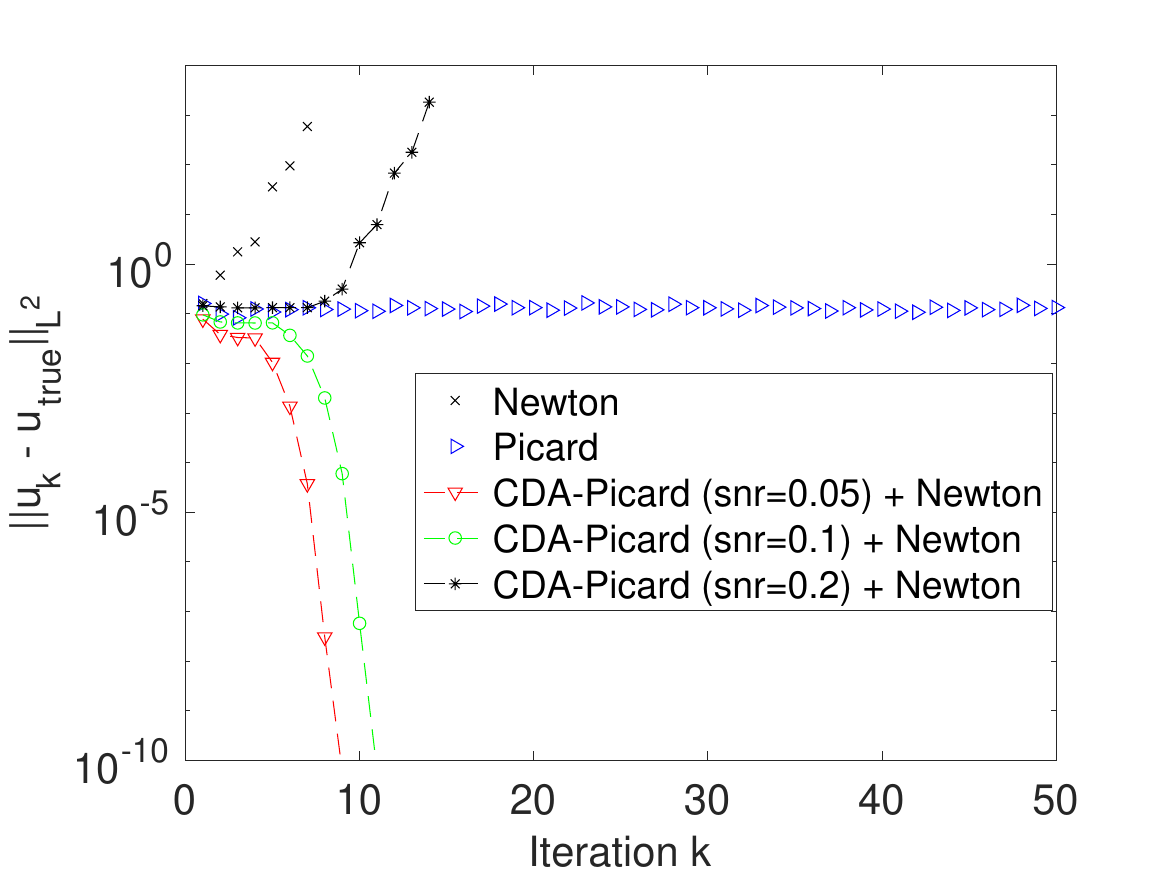}
\caption{\label{convplotsN} Shown above are $L^2$ error vs iteration number, for the 2D driven cavity at  $Re$=10000 (left) and the 3D driven cavity at $Re$=1000 (right) for Newton, Picard, and CDA-Picard + Newton.}
\end{figure}

\section{Conclusions and future directions}

We have extended the CDA-Picard methodology in this paper.  We have improved the analysis by providing convergence results in more appropriate norms than the original analysis provides, and also extended the methodology and analysis to the more realistic case of noisy data.  Our analytical and numerical results show that the $L^2$ error scales with the relative size of the noise.  For large relative noise levels, however, we showed that using the CDA-Picard iteration to generate initial guesses for the usual Newton iteration can be very effective.

For future work, we plan to consider applications of CDA-Picard to multiphysics problems related to the NSE.  Additionally, we plan to look for methods to reduce the amount of data required by the method, in particular in the 3D case.

\section{Data availability}

Data will be made available on request.

\section{Declaration of competing interest}

The authors declare the following financial interests/personal relationships which may be considered as potential competing interests: Leo Rebholz reports financial support was provided by National Science Foundation.  Bosco Garc{\' i}a-Archilla and Julia Novo report financial support from MCIN/AEI of Spain and the European Union.

%\bibliographystyle{plain}
%\bibliography{graddiv,Xuejian_Li_ref,references}

\section{Appendix}

\begin{proof}[Proof of Theorem \ref{rescon1}]
Subtracting the iteration at Step $k+1$ from Step $k$, and testing with $u_{k+1} - u_k$ gives
\begin{equation}\label{step1}
\nu \| \nabla (u_{k+1} - u_k ) \|^2 + \mu \| I_H (u_{k+1} - u_k) \|^2 \le -b(u_k - u_{k-1},u_k,u_{k+1}-u_k),
\end{equation}
thanks to $b(u_{k-1},u_{k+1} - u_k,u_{k+1}-u_k)=0$.  Using \eqref{bbound2} and Young's inequality, we obtain
\[
\nu \| \nabla (u_{k+1} - u_k ) \|^2 + \mu \| I_H (u_{k+1} - u_k) \|^2 \le M \| \nabla( u_k - u_{k-1}) \| \| \nabla u_k \| \| \nabla ( u_{k+1}-u_k) \|.
\]
Using Lemma \ref{lem2} (for sufficiently large $k$) on the first right hand side term and dropping the second (positive) left hand side term, we get the stated bound 
\[
\| \nabla (u_{k+1} - u_k ) \|  \le  C_{\delta}  \alpha   \| \nabla( u_k - u_{k-1}) \|.
\]
\end{proof}

\begin{proof}[Proof of Theorem \ref{rescon2}]
Starting from \eqref{step1} and using \eqref{bbound1} and Young's inequality, we obtain
\[
\nu \| \nabla (u_{k+1} - u_k ) \|^2 + \mu \| I_H (u_{k+1} - u_k ) \|^2 \le M \| \nabla( u_k - u_{k-1}) \|^{1/2} \| u_k - u_{k-1} \|^{1/2} \| \nabla u_k \| \| \nabla ( u_{k+1}-u_k) \|.
\]
Proceeding as in Lemma \ref{lem1} to lower bound the left hand side and using Lemma \ref{lem2} (for sufficiently large $k$) on the first right hand side term along with Young's inequality gives
\[
\frac{\nu}{4} \| \nabla (u_{k+1} - u_k ) \|^2 + \hat \lambda \| u_{k+1} - u_k \|^2 \le {C_{\delta}^2 \nu \alpha^2} \| \nabla( u_k - u_{k-1}) \| \| u_k - u_{k-1} \|.
\]
Applying Young's inequality again to the right hand side gives
\begin{align*}
\frac{\nu}{4} \| \nabla (u_{k+1} - u_k ) \|^2 + \hat \lambda \| u_{k+1} - u_k \|^2 
& \le \rho \frac{\nu}{4} \| \nabla( u_k - u_{k-1}) \|^2 +  \frac{C_{\delta}^4 \nu \alpha^4}{ \rho} \| u_k - u_{k-1} \|^2 \\
& \le \rho \left( \frac{\nu}{4} \| \nabla( u_k - u_{k-1}) \|^2 +  \frac{C_{\delta}^4 \nu \alpha^4}{ \rho^2 } \| u_k - u_{k-1} \|^2 \right).
\end{align*}
Using the stated choice of $\rho$ and assumption that it is less than 1, we obtain the stated result,
\begin{align*}
\frac{\nu}{4} \| \nabla (u_{k+1} - u_k ) \|^2 + \hat \lambda \| u_{k+1} - u_k \|^2 
\le
2C_I C_{\delta}^2 \alpha^2 H \left( \frac{\nu}{4} \| \nabla( u_k - u_{k-1}) \|^2 +  \hat \lambda \| u_k - u_{k-1} \|^2 \right).
\end{align*}

\end{proof}

%	While the Newton iteration is also popular for solving \eqref{NS} due to its quadratic convergence, Newton requires a good initial guess and is often used together with the Picard iteration.  A common strategy is using Picard until sufficiently close to the solution and then switching to Newton.  The Newton iteration for the NSE is given by {\color{red} (\cite{J16} p. 339)}
%	\begin{align*}
%		-\nu \Delta u_{k+1}+u_k\cdot\nabla u_{k+1}+ u_{k+1}\cdot\nabla u_k + \nabla p_{k+1} &={f} + u_k \cdot\nabla u_k, \\
%		\nabla\cdot {u}_{k+1}&=0.
%	\end{align*}
%	
	
%	and CDA-Newton iteration
%	\begin{align*}
%		-\nu \Delta u_{k+1}+u_k\cdot\nabla u_{k+1}+ u_{k+1}\cdot\nabla u_k + \nabla p_{k+1} + \mu I_H(u_{k+1} - u)&={f} + u_k \cdot\nabla u_k, \\
%		\nabla\cdot {u}_{k+1}&=0.
%	\end{align*}

\end{document}